\documentclass{article}

\usepackage{amssymb,authblk}
\newcommand{\real}{{\mathbb R}}

\newcommand{\nat}{{\mathbb N}}
\newcommand\sF{\mathcal{F}}
\newcommand\sY{\mathcal{Y}}

\newcommand{\integ}{{\mathbb Z}}
\newcommand{\posint}{{\mathbb Z}_+}
\newcommand{\E}{{\mathbb E}}
\newcommand{\bE}{{\mathbb E}}

\usepackage{amsmath,amsfonts,theorem,epsfig,amssymb}

\usepackage[T1]{fontenc}
\usepackage{mathptmx}
\usepackage{adjustbox}
\usepackage{xcolor}

\usepackage{algpseudocode}
\usepackage{algorithm}
\usepackage{verbatim}

\newtheorem{lemma}{Lemma}
\newtheorem{proposition}{Proposition}
\newtheorem{theorem}{Theorem}
\newtheorem{corollary}{Corollary}
\newtheorem{remark}{Remark}

\newtheorem{assumption}{Assumption}
\newcommand{\qed}{\mbox{\rule{.4em}{1.7ex}\hspace{.6em}}}
\newenvironment{proof}{{\bf Proof\ }}{\hspace*{.1em}\hfill\qed
\bigskip \noindent}

\begin{document}




\title{Stochastic Filtering of Reaction Networks Partially Observed in Time Snapshots}

\author[1]{Muruhan Rathinam}
\author[2]{Mingkai Yu}
\affil[1]{Department of Mathematics and Statistics, University of Maryland Baltimore County}
\affil[2]{Department of Mathematics and Statistics, University of Maryland Baltimore County}

\maketitle
\begin{abstract}
Stochastic reaction network models arise in intracellular chemical reactions, epidemiological models and other population process models, and are a class of continuous time Markov chains which have the nonnegative integer lattice as state space.  
We consider the problem of estimating the conditional probability distribution of a stochastic reaction network 
given exact partial state observations in time snapshots. 
We propose a particle filtering method called the {\em targeting method}. Our approach takes into account that the reaction counts in between two observation snapshots satisfy linear constraints and also uses inhomogeneous Poisson processes as proposals for the reaction counts to facilitate exact interpolation. We provide rigorous analysis as well as numerical examples to illustrate our method and compare it with other alternatives.     
\end{abstract}







\section{Introduction}\label{sec-intro}
Stochastic reaction network models describe a wide array of phenomena such as intracellular biochemical reactions, epidemiological models, predator prey systems etc. These models are a particular form of continuous time Markov chains (CTMCs) $Z(t)$ that have the multidimensional nonnegative integer lattice $\posint^n$ as state space and have finitely many types of ``reaction'' events $j=1,\dots,m$ which change the state by a fixed amount $\nu_j \in \integ^n$. Given the wide applicability of these models, the problem of inferring the state as well as parameters from partial observations is of critical interest. 

In this paper, we are focused on stochastic reaction networks where some of the states, say $Y(t)$ where $Z(t)=(X(t),Y(t))$, are observed exactly at time snapshots $t_1 < t_2 < \dots < t_{T_s}$. We are interested in estimating the conditional probability
\[
\pi(t,z) = P\{Z(t) = z \, | Y(t_l)=y_l, l=1,\dots,T_s\}
\]
via a suitable filtering algorithm. More generally, our methods can be used to estimate the conditional expectation of 
a path functional given a sequence of observations.

Starting from the celebrated Kalman filter \cite{kalman1960new} for discrete time linear systems with Gaussian noise, several filtering methods have been developed for various dynamic models \cite{kalman1961new, doucet2001sequential, del2006sequential, fristedt2007filtering, golightly2006bayesian, golightly2011bayesian, calderazzo2019filtering, fang2022stochastic}. In linear and Gaussian models the conditional distributions themselves are Gaussian and are characterized by the mean and the covariance. This means that it is possible to obtain (deterministic) evolution equations for the conditional mean and covariance which makes the computations relatively straightforward. However, when the dynamic models are nonlinear or if the noise in the dynamics are non-Gaussian, the situation is no longer simple. While deterministic evolution equations may be obtained for the conditional probabilities, these equations are infinite dimensional or very high dimensional. This necessitates Monte Carlo methods. Particle filtering, also known as Sequential Monte Carlo (SMC), are usually the most 
practical approach for such problems since these methods enable recursive computations as new observations are added \cite{doucet2001sequential, del2006sequential, bain2008fundamentals}. 
Particle filters can also be used for Bayesian estimation of parameters by treating parameters as states that are constant in time, thereby providing a unified approach for state and parameter estimation. As such, our focus in this paper is on the state estimation. If the goal is only to estimate parameters, one may approach the problem in a non-Bayesian setting and several methods exist in the literature \cite{liao2015tensor, horvath2008parameter}. 

Particle filtering methods involve simulating identically distributed proxies $\tilde{Z}^{(i)}(t)$ (which we call the ``filter processes'') of the reaction network along with importance weights $W^{(i)}(t)$ for $i=1,\dots,N_s$. The fidelity of the filter estimate depends on $N_s$ as well as the variance of $W^{(i)}(t)$. Usually, the variance of $W^{(i)}(t)$ tends to increase in time, leading to poor fidelity of the estimate. Thus most methods are concerned with coming up with ways to slow down the growth of variance. 

We briefly describe some of the results in the literature for particle filtering of stochastic reaction networks. The work in  \cite{golightly2019efficient} considers partial state observations in discrete time with added Gaussian noise 
and proposes particle filtering methods where the proxies $\tilde{Z}^{(i)}$ are simulated according to the conditional probability of a Gaussian approximation of $Z$. \cite{fang2020stochastic, fang2022stochastic} are concerned with situations where some molecular copy numbers are very large so that a hybrid approximation of the process $Z$ may be employed and provide rigorous mathematical analysis to justify the method. The first of these is concerned with discrete time observations while the second is concerned with both discrete time as well as continuous time observations. A related work  \cite{fang2023convergence} studies the convergence of a regularized form of the particle filter. 
All of the above works are concerned with noisy observations where the observation noise is additive Gaussian even though the reaction network model is not.  

Our previous work \cite{rathinam2021state} provides a particle filter for a reaction network where exact (noiseless) partial state observations are made in continuous time. To our best knowledge this is the first such filter for a CTMC. In this paper, we provide new particle filtering methods for the case of exact partial state observations in discrete time snapshots. Since we are concerned with noiseless observations, the filters mentioned previously (which apply to observations with Gaussian noise) are not suitable. At first glance, noiseless observations may seem unrealistic. However, noisy observation processes themselves could be included in the reaction network model as additional reactions. For instance, if we observe a fluorescently tagged molecular species $S$, then the actual observations count the copy number of photons $P$ emitted by $S$. Hence this measurement process may be modeled by adding a reaction $S \rightarrow S + P$ and we consider $P$ as the observed species. 

The conceptually simplest approach to filtering for discrete time observations (with or without noise) involves the ``prediction/correction'' approach. If we denote by $\pi_{l,j}(z)$ the conditional probability distribution of $Z(t_l)$ given observations up to time point $t_j$, then one computes $\pi_{l,l}$ from $\pi_{l-1,l-1}$ in two steps. The first step is the prediction step which simply evolves the proxy process $\tilde{Z}$ according to the Markovian evolution of $Z$ to obtain weighted samples 
from $\pi_{l,l-1}$ and then a correction is performed by altering the weights appropriately so that the weighted sample reflects $\pi_{l,l}$. In the noiseless case, this approach will result in most of the weights set to zero. Even if one modifies the 
law of the proxy $\tilde{Z}$ via Girsanov changes of intensity, the probability of satisfying the observation at the next observation time $t_l$ may be negligibly small. The method we propose in this paper forces the process $\tilde{Z}$ to satisfy the next observation
and we call this the {\em targeting method}. In order to target the next observation we work with a Girsanov change of measure under which the reaction counts are inhomogeneous Poissons. We provide rigorous justification of our method and illustrate its superior performance compared to other alternatives.  

The rest of the paper is organized as follows. Section \ref{sec-RN} provides a brief review of stochastic reaction networks and states the filtering problem addressed in this paper. In Section \ref{sec-filter} we discuss the mathematical setup, motivate our method which we call the {\em targeting algorithm} as well as discuss other ideas explored in the literature. Section \ref{sec-examples} first describes the reaction network examples that are later used in Section \ref{sec-eg-discrete-time} to illustrate various choices within our method and compare with other methods. Finally, in Section \ref{sec-conclusion}, we offer some concluding remarks. 

\section{Reaction Networks and The Filtering Problem}\label{sec-RN}
We describe the widely used stochastic model of a {\em reaction network}.
A reaction network consists of $n \in \nat$ {\em species} and $m \in \nat$ {\em reaction channels}. The copy number of the species $i$ ($1 \leq i \leq n$) at time $t \geq 0$ is denoted by $Z_i(t) \in \posint$ and the number of firings of reaction channel $j$ ($1 \leq j \leq m$) during $(0,t]$ is denoted by $R_j(t) \in \posint$. The vector valued processes $X$ and $R$ are defined by $Z(t)=(Z_1(t),\dots,Z_n(t))$ and $R(t)=(R_1(t),\dots,R_m(t))$ and $Z$ is assumed to be a continuous time Markov chain (CTMC) with statespace $\posint^n$.  
The characteristics of a reaction network are uniquely defined by the {\em stoichiometric vectors} $\nu_j \in \integ^n$ and the {\em propensity functions} $a_j:\posint^n \to [0,\infty)$ for $j=1,\dots,m$. 
If the reaction channel $j$ fires at time $t$, the state $Z$ of the process is changed by $\nu_j$.   We assume that the process $Z$ is {\em cadlag}, that is, right continuous with left hand limits. We shall define the matrix $\nu \in \integ^{n \times m}$ so that its $j$th column equals $\nu_j$. The $(i,j)$th entry of $\nu$ is denoted $\nu_{ij}$ which is the $i$th entry of the column vector $\nu_j$.  Thus we may write
\begin{equation}
Z(t)=Z(0) + \nu \, R(t).    
\end{equation}
The propensity function $a_j$ has the interpretation that 
given $Z(t)=z$ (for $z \in \posint^n$), the probability that the reaction channel $j$ fires during $(t,t+h]$ is given by $a_j(z) h + o(h)$ as $h \to 0+$. It follows that $R$ is an $m$-variate {\em counting process} with {\em stochastic intensity} $a(Z(t-))$ \cite{bremaud}. We assume that the processes $Z$ and $R$ are carried by a probability space $(\Omega,\sF,P)$. 

We shall be concerned with the situation where a subset of the species are observed at time snapshots $0 \leq t_1 <t_2 \dots$. We shall take the vector of observed species to consist of the last $n_2$ components of the state $Z(t)$ and denote it by $Y(t)$. The vector of the rest of the unobserved species shall be denoted by $X(t)$. Thus we write $Z(t)=(X(t),Y(t))$ where $X(t) \in \posint^{n_1}$ and $Y(t) \in \posint^{n_2}$ with $n=n_1+n_2$. We shall write $\nu_j=(\nu_j',\nu_j'')$ where $\nu_j' \in \integ^{n_1}$ and $\nu_j'' \in \integ^{n_2}$. We shall define $\nu' \in \integ^{n_1 \times m}$ and $\nu'' \in \integ^{n_2 \times m}$ to consist of the first $n_1$ rows of $\nu$ and the last $n_2$ rows of $\nu$ respectively. Thus we may write
\begin{equation}
    \begin{aligned}
        X(t) &= X_0 + \nu' R(t),\\
        Y(t) &= Y_0 + \nu'' R(t),
    \end{aligned}
\end{equation}
where $Z_0=(X_0,Y_0)$ is the random initial state. 


The filtering problem that we are concerned with is the numerical computation of conditional probabilities of the form: 
\begin{equation}
    \pi(t,z) = P\{Z(t)=z \, | \, Y(t_i)=y_i, \, i=1,\dots,T_s\}
\end{equation}
where $t \geq 0$, $z \in \posint^{n}$ and $T_s$ is the number of snapshots.  
Due to the fact that the number of states is either infinity or exceedingly large, Monte Carlo methods are the most feasible way to compute $\pi(t,z)$. In addition to the observation, we assume that we know the distribution of $Z_0=(X_0,Y_0)$. We note that $\pi(t,z)$ depends also on the observation sequence 
$\bar{y}=(y_1,\dots,y_{T_s})$ so that sometimes we write $\pi(t,z,\bar{y})$ to emphasize this dependence. See Remark \ref{rem-realistic} for some details.

In a state inference problem, it would be important to know in advance how useful our observations would be for the prediction of the state. For instance if the full process is $Z=(X,Y)$ where $X$ and $Y$ are independent, then the observation of $Y$ will not be of any use in the prediction of $X$. In that case, our method, or any other 
filtering method will simply end up providing an estimate of the unconditional distribution of $X$, 
because both the conditional distribution $\pi(t,x)$ and the unconditional distribution $p(t,x)$ will be equal. 
One measure of the difference between $\pi(t,x)$ and $p(t,x)$ is the {\em mutual information}. 
In the context of stochastic reaction networks, the literature on estimating mutual information is limited and some recent results may be found in \cite{duso2019path, moor2023dynamic}. 
If our goal is to obtain a point estimate for the unknown states vector $X(t)$ given 
an observation sequence denoted by $\bar{Y}$ (here $\bar{Y}$ is vector of snap shots $Y(t_l)$), the best estimate of $X(t)$ is the conditional expectation $\E[X(t) \, | \bar{Y}]$. The inherent uncertainty in this  
estimate is given by the conditional variance $\text{Var}(X(t) \, | \, \bar{Y})$. Both these quantities can be estimated from a filtering algorithm and can be used to obtain a confidence interval for the estimate. Then the expected conditional variance $\E( \text{Var}(X(t) \, | \bar{Y}) )$ (where the expectation is over all possible observations) provides a quantification of the error in using 
the specific observation model based on snapshots of $Y$ to predict $X(t)$. By the conditional variance formula, one obtains
\[
\text{Var}(X(t)) = \text{Var}(\E(X(t) \, | \, \bar{Y})) + \E(\text{Var}(X(t) \, | \, \bar{Y})).
\]
The larger the first term on the right the smaller the second term would be and hence the more useful it would be to observe $\bar{Y}$. A priori quantification of these terms for reaction networks could be a useful direction of research. 
A more detailed discussion in the context of continuous time observations may be found in \cite{rathinam2021state}. 

Finally, we also remark that inference of state $Z(t)$ at a much earlier time $t$ based on observations far out in the future (at times $t_i$ where $t_i \gg t$) is in general, a numerically challenging problem when particle filtering methods are used. This is due to the fact that the importance weights degenerate as time progresses and the standard remedy for this is resampling (see \cite{bain2008fundamentals, rathinam2021state} for instance). During a resampling procedure, each sample/particle in the collection is replaced by a number of ``offsprings'' which may be zero for some particles. If the time span between $t$ and $t_i$ is large, multiple resamplings would be needed and very few particles ``alive'' at time $t$ may have a `descendent' at time $t_i$,  resulting in a low fidelity estimate of the 
state $Z(t)$ given the future observation. A remedy for this may be found in the {\em backward smoothing algorithms} (see \cite{dau2023backward} for instance) which we shall not explore in this paper.

\begin{remark}\label{rem-realistic}
Based on the preceding discussions, we surmise that even though our filtering methods may be used to predict a state at a past time $t$ based on future observations $t_i$ where $t_i \gg t$, a large sample size (number of particles) would be required for a good estimate. Hence, the most useful application of our filtering methods would be for the estimation of a conditional probability of the form  
\begin{equation}\label{eq-pi-filter-gen}
\pi(t,z;(y_1,\dots,y_l)) = P\{Z(t)=z \, | \, Y(t_i)=y_i, \;\; i=1,\dots,l\}, 
\end{equation}
where $t_l - t$ is not too large compared to the time scales of the system. A precise quantification of the time scales of the system itself requires assumptions. For instance, if the system is an ergodic Markov chain, then the time scale would be the mixing time. 
\end{remark}

\section{The Filtering Algorithm}\label{sec-filter}
This section is organized as follows. In Section \ref{sec-math-setup} we provide a mathematical framework to describe filtering approaches in general. Then in Section \ref{sec-naive} we describe the most obvious approach (which we call the ``naive algorithm'') to the filtering problem at hand. In Section \ref{sec-ESS}, we discuss the important concept of ``effective sample size'' which also highlights the main drawback of the naive algorithm. In Section \ref{sec-cond-prop} we introduce the important concept of a ``conditional propensity function'', which underpins some of the existing methods in the literature. In Section \ref{sec-targeting}, we introduce the main idea of our proposed filtering algorithm which we call the {\em Targeting Algorithm}. This focuses on 
conditioning on a partial observation $Y(T)=y$ at a single snapshot in time.  Finally, in Section \ref{sec-multiple} we briefly outline how the Targeting Algorithm can be applied iteratively 
in the case of observations at multiple snapshots in time.

\subsection{Mathematical Setup}\label{sec-math-setup}
We shall be concerned with weighted Monte Carlo methods which are also known as {\em importance sampling} or {\em particle filters}. We present a mathematical description of a general  setup as follows. 
We let $(\Omega, \sF)$ be a measurable space consisting of the sample space and a collection of events. 

It is useful to think of two sets of processes: one on which observations are made and the other consists of the Monte Carlo simulations. We refer to the former as the {\em lab processes} and the latter as the {\em method processes} or {\em filter processes}. Both sets of processes are carried by $(\Omega, \sF)$. Conditioned on the observations, the lab processes and the method processes are independent. 

The lab processes consist of $Z=(X,Y)$ and $R$, where
$Z:[0,T_0] \times \Omega \to \posint^n$ and $R:[0,T_0] \times \Omega \to \posint^m$ are measurable, and $T_0>0$ is large enough so that $[0,T_0]$ contains all observation time points $t_l$.   

The method processes consist of $\tilde{Z}^{(i)}=(U^{(i)},V^{(i)}):[0,T_0] \times \Omega \to \posint^n$,  
$I^{(i)}:[0,T_0] \times \Omega \to \posint^m$ and $W^{(i)}:[0,T_0] \times \Omega \to [0,\infty)$ for $i=1,\dots,N_s$. Here $N_s$ is the {\em filter sample size} (number of particles) and the method processes for different $i$ are all identically distributed and for convenience we drop the superscript $i$ when describing them, or one may regard $\tilde{Z}=(U,V)$, $I$ and $W$ as the template method processes. The method processes $U,V$ and $I$ are proxies for $X,Y$ and $R$ respectively. The process $W$ is the weight process that assigns a degree of importance to the paths of $U,V$ and $I$. 

Often more than one probability measure on $(\Omega,\sF)$ is considered. The lab processes have the given propensities under a probability measure $P$ while under 
another probability measure $P_0$ (with $P \ll P_0$) both the lab and filter processes have a convenient description. In general, the quantity of interest 
$\pi(t,z,\bar{y})$ is given as a function of the conditional expectation of path functionals of the filter processes.
Going forward we denote the observation event sequence by $O_{\bar{y}}$. 


\subsection{The Naive Algorithm}\label{sec-naive}
It might be clear from the outset that the most straightforward Monte Carlo algorithm is simply based on the following prediction/correction strategy without any changes of measure. One simulates $N_s$ sample trajectories  
$\tilde{Z}^{(i)}(t)=(U^{(i)}(t),V^{(i)}(t))$ for $i=1,\dots,N_s$ according to the given reaction network characteristics via the Gillespie algorithm or a variant \cite{Gillespie77, gillespie2007stochastic, gibson2000efficient,anderson2007modified} and assigns a weight $W^{i}=1$ (accept) if and only if $V^{(i)}(t_l)=y_l$ for $l=1,\dots,T_s$ and $W^{i}=0$ (reject)  otherwise. We note that $\pi(t,z,\bar{y})$ is given by 
\[
\pi(t,z,\bar{y}) = \frac{\E[1_{\{z\}}(\tilde{Z}(t) \, W \, | \, O_{\bar{y}}]}{\E[W \, | \, O_{\bar{y}}]}
\]
and one simply estimates $\pi$ via the filter sample as

\[
\hat{\pi}(t,z,\bar{y}) = \frac{\sum_{i=1}^{N_s} 1_{\{z\}}(\tilde{Z}^{(i)}(t)) \, W^i} {\sum_{i=1}^{N_s}  W^i}. 
\]
We note that the weights $W^i$ depend on $\bar{y}$. 

We shall refer to this method as the {\em Naive Algorithm}. While it is conceptually simple and easy to implement, in many situations the probability of any given sequence of observation snapshots is very low and hence the ``effective sample size'' is much smaller than $N_s$. Hence in order to obtain reliable estimates one needs to simulate impractically large numbers $N_s$ of trajectories. 
However, the naive algorithm has one advantage in that it is conditionally unbiased:
\[
\E \left( \frac{1}{\sum_{i=1}^{N_s}  W^i} \sum_{i=1}^{N_s} 1_{\{z\}}(\tilde{Z}^{(i)}(t)) \, W^i \, \Big{|} \, \sum_{i=1}^{N_s} W^i \neq 0, \, O_{\bar{y}} \right) = \pi(t,z,\bar{y}). 
\]

\subsection{Effective Sample Size}\label{sec-ESS}
Often the filtering method involves estimating a quantity $c$ given by
\[
c = \frac{\E[C \, W \, | \, O_{\bar{y}}]}{\E[W | \, O_{\bar{y}}]}
\]
where $C$ is a path functional and $W$ is an associated weight. The quantity $c$ is estimated via 
the empirical mean 
\[
\hat{c} = \frac{\sum_{i=1}^{N_s} C^i W^i}{\sum_{i=1}^{N_s} W^i}
\]
from a weighted sample $(C^{i},W^i)$ for $i=1,\dots,N_s$ identically distributed like $(C,W)$.
The fidelity of this estimate 
depends on the joint (conditional) distribution of $(C^i,W^i)$ for $i=1,\dots,N_s$. Law of large numbers (which applies under modest assumptions) states that for sufficiently large $N_s$, $\hat{c}$ is a good approximation of $c$. If we assume conditional independence (for different $i$), when $N_s$ is sufficiently large, under modest assumptions on moments we may use the so-called delta method \cite{billingsley2017probability} and obtain the approximation 
that $\hat{c}$ is (approximately) normally distributed  
with $\E[\hat{c} \, | \, O_{\bar{y}}] \approx c$ and 
\[
\text{Var}(\hat{c} \, | \, O_{\bar{y}}) \approx \frac{\E^2[CW \, | \, O_{\bar{y}}]}{N_s \, \E^2[W \, | \, O_{\bar{y}}]} \,  \text{Var}\left(\frac{CW}{\E[CW \, | \, O_{\bar{y}}]}-\frac{W}{\E[W \, | \, O_{\bar{y}}]}  \Big{|} O_{\bar{y}}\right).
\]
The square of the {\em relative error} $\text{RE}=\frac{\sqrt{\E[(\hat{c}-c)^2 \, | O_{\bar{y}}]}}{c}$ is then approximated by 
\[
\text{RE}^2 \approx \frac{1}{N_s}\text{Var}\left(\frac{CW}{\E[CW \, | \, O_{\bar{y}}]}-\frac{W}{\E[W \, | \, O_{\bar{y}}]}\Big{|} O_{\bar{y}} \right).
\]
The fidelity of the estimate clearly depends on the path functional $C$ and its conditional covariance with $W$ among other factors. For instance, if one is interested in estimating $\pi(t,z)$, then $C=1_{\{z\}}(\tilde{Z}(t))$, and the fidelity of the estimate varies with $z$. In order to have a measure of fidelity that is independent of the path functional $C$ and solely dependent on $W$, the quantity 
\begin{equation}
    N_e = \frac{\left(\sum_{i=1}^{N_s} W^i\right)^2}{\sum_{i=1}^{N_s} (W^i)^2}
\end{equation}
known as {\em effective sample size} has been proposed in the literature \cite{liu1996metropolized}. When $N_s$ is very large
\begin{equation}
    N_e \approx N_s \, \frac{\E^2[W | O_{\bar{y}}]}{\E[W^2 | O_{\bar{y}}]}.
\end{equation}
Thus, for the naive method, $N_e \approx N_s \, p$ where 
$p$ is the probability of the sequence of observations. 

\begin{remark}\label{rem-naive-ess}
When the molecular copy numbers are in the order of hundreds and if multiple species are observed,
the probability of any observed value will be very small and hence the effective sample size
of the naive method will be very small.
\end{remark}

\subsection{Conditional Propensity Function} \label{sec-cond-prop}
Recall that the propensity $a_j(z)$ is defined by the prescription that the conditional probability of 
one reaction event of $j$ occurs during $(t,t+h]$ conditioned on $Z(t)=z$ is given by $a_j(z) h + o(h)$ as $h \to 0+$. 
We may define the {\em conditional propensity} $b_j(z,t,t_1,y_1\dots)$ 
of the reaction channel $j$ by replacing the conditioned event $\{Z(t)=z\}$ by the event 
\[
\{Z(t)=z, Y(t_l)=y_l \;\; l=1,\dots,T_s\}
\]
in the above definition so that it includes the observations. See \cite{fearnhead2008computational, golightly2019efficient} for instance. 
Suppose that $t_{l^*-1} < t \leq t_{l^*}$. Then using the Markov property, it can be shown that 
\begin{equation}
  b_j(z,t) 
  = a_j(z) \frac{P(Y(t_l)=y_l, l \geq l^* \,|\, Z(t)=z+\nu_j)}{P(Y(t_l)=y_l, l \geq l^*\,|\, Z(t)=z)},
\end{equation}
where we have suppressed the dependence on the observation sequence $(t_l, y_l)$.
We note that due to the explicit dependence on time $t$, the (conditional) wait time between reaction events is no longer exponential. If an explicit formula for $b_j(z,t)$ is available and the wait times have an analytically tractable formula (as we shall see in a specific example) it is indeed possible to simulate the reaction network according to these conditional propensities. 
In that event, all trajectories will satisfy the observations. 

In general, an explicit formula for $b_j$ is not available and hence it is impossible to draw samples from the conditional distribution. However, if a good approximation of $b_j$ is proposed, which we denote by $\hat{a}_j(z,t)$, then one may simulate the reaction network according to these ``proposal propensities''. Since the proposal propensities may not be exact, a weight $\tilde{L}_{\bar{y}}(T)$ given by the Girsanov change of measure that corresponds to the change in propensity from  $\hat{a}_j$ to $a_j$ (for the method process $I$) is computed (from the trajectory $\tilde{Z}$) and applied. Additionally, the indicator weight $W$ as in the naive method is applied depending on whether the observations are satisfied or not, so that the overall weight is $W \tilde{L}_{\bar{y}}(T)$. We note that the change of measure weight is observation dependent and hence the subscript $\bar{y}$ denoting the observation sequence.  

If the proposal propensities $\hat{a}_j$ are exactly equal to the conditional propensities $b_j$, then with probability one, all trajectories $\tilde{Z}$ will satisfy the observations and 
$W \tilde{L}(T)$ will be a constant, making the effective sample size $N_s$. If $\hat{a}_j$ are close enough to $b_j$, then one expects 
the effective sample size to be close to $N_s$. This is the strategy that underlies the methods in \cite{golightly2019efficient}.

\subsection{The Targeting Problem and the Targeting Algorithm}\label{sec-targeting}

\setlength{\tabcolsep}{10pt}
\renewcommand{\arraystretch}{1.5}
\begin{table}
    \centering
    \begin{tabular}{|c|c|}
    \hline
      Original (Lab)   &  Proxy (Filter)\\
      \hline \hline
       $Z_0=(X_0,Y_0)$     &  $\tilde{Z}_0 = (U_0,V_0)$ \\
       \hline
       $R(T)=(R'(T),R''(T))$  & $K=(K',K'')$ \\
       \hline 
        $R(t)$  & $I(t)$ \\
        \hline
    \end{tabular}
    \caption{The notations for the original (lab) random variables and processes (on the left) and their filter proxies (on the right).}
    \label{tab:lab-filter}
\end{table}

The new method proposed by us is based on the solution of what we call a {\em targeting problem}, which we define first. In a targeting problem, we are given a time interval $[0, T]$, a distribution $\mu$ for the initial random state $Z_0$, and a partial final state observation value $Y(T)=y$. The goal is to compute the conditional distribution $P\{Z(t)=z \, | \, Y(T)=y\}$ where $t \in [0,T]$. 
Our approach to this targeting problem is based on two key ideas. The first is that the partial final state constraint $V(T)=y$ can be translated into linear relations among the reaction counts. The second is that for a (possibly inhomogeneous) Poisson process if the value at time $T>0$ is known, then one may simulate exact conditional realizations of its path on $[0,T]$. The Targeting Algorithm is summarized in Algorithm \ref{alg-target}. Table \ref{tab:lab-filter} summarizes the notations for the various random variables and processes. 

In our targeting approach, we fix an observation value $y$ and assume that the reaction count processes $R_j$ are independent inhomogeneous Poissons with strictly positive  
intensities $\lambda_j(Z_0,t,y)$ (when conditioned on $Z_0$) under a probability measure $P_0$.
Throughout this paper, by a $d$-variate inhomogeneous Poisson process $J(t)$ with a given $d$-dimensional vector intensity function $\alpha(t)$, we shall mean that the components $J_i(t)$ are independent inhomogeneous Poisson processes with corresponding intensity functions $\alpha_i(t)$. Likewise, a $d$-variate Poisson random variable with $d$-dimensional mean vector shall stand for a random vector with independent components which are Poissons with the corresponding mean components.  

\subsubsection{First step of the Targeting Algorithm} The first  
step of the targeting method involves generating the proxies $\tilde{Z}_0$ (for the initial state $Z_0$) and $K$ (for the total reaction count vector $R(T)$ over the interval $[0,T]$), along with a weight $W_p$ so that the following holds:
\[
P_0\{Z_0=z_0, R(T)=k \, | \, Y(T)=y\} = \frac{E_0[ 1_{\{(z_0,k)\}}(\tilde{Z}_0,K) \, W_p \, | \, Y(T)=y]}{E_0[W_p \, | \, Y(T)=y]}.
\]
This is accomplished based on the fact that the observed event $\{Y(T)=y\}$ can be equivalently represented by
a linear relationship among the components of reaction count vector $R(T)$. In fact, after reordering and splitting the reaction count vector $R(T)=(R'(T),R''(T))$ along with possible removal of dependent observable species (corresponding to dependent rows of $\nu''$), we may write the modified $\nu''$ as 
\begin{equation}
    \nu'' = [A \; \; B]
\end{equation}
where $B$ is invertible. 
This yields
\[
y=Y_0 + A R'(T) + B R''(T)
\]
and hence 
\begin{equation}\label{eq-Rpp}
R''(T) = B^{-1} (y-Y_0-A R'(T)).
\end{equation}
Thus, we see that $R''(T)$ is determined by $Y_0, y$ and $R'(T)$. We note that the matrix $B$ is $n_2 \times n_2$ where $n_2$ is the effective number of observed species. It follows that the vector $R''(T)$ is of size $m_2=n_2$ while $R'(T)$ is of size $m_1=m-m_2$. We shall make the proxies $\tilde{Z}_0=(U_0,V_0)$ and $K$ satisfy this
condition. In particular, we shall require 
\begin{equation}\label{eq-Ipp}
K'' = B^{-1} (y-V_0-A K'),
\end{equation}
here $K'$ and $K''$ are proxies for $R'(T)$ and $R''(T)$ respectively.
If we choose $K''$ according to \eqref{eq-Ipp} we automatically satisfy the target observation. 
In order to correct for this assignment, we assign the weight 
\begin{equation}\label{eq-Wp}
W_p = \rho(K'',\tilde{Z}_0,y),
\end{equation}
where $\rho(k'',z_0,y)$ is the probability that 
an $m_2$-variate Poisson random variable with mean $\int_0^T \lambda''(t,z_0,y) dt$ equals $k'' \in \real^{m_2}$. Here the vector of intensities $\lambda(t,z_0,y)$ is split into $\lambda=(\lambda',\lambda'')$ in the obvious way.  
We observe that if $K''$ is not a non-negative integer vector, then $W_p=0$. We can think of this procedure as a prediction-correction step and we have traded the use of an indicator weight with this Poisson weight (see Remark \ref{rem-poisson-weight}).

This first step of the targeting algorithm, which we call the End Point Algorithm, is described in Algorithm \ref{alg-endpoint}. The first step of the end point algorithm is to generate a sample $\tilde{Z}_0=(U_0,V_0)$ according to $\mu$ (the given initial distribution) and then a Poisson vector $K'$ with mean $M'=\int_0^T \lambda(t,\tilde{Z}_0,y) dt$. Then we set $K''$ according to \eqref{eq-Ipp}. If $K''$ is not a vector of non-negative integers then we repeat this step because 
it would yield an importance weight of $W_p=0$ and there is no reason to use this pair $(\tilde{Z}_0,K)$ in the second step (interpolation) of the targeting algorithm. If $K''$ is a vector of non-negative integers, then the weight $W_p$ is computed according to \eqref{eq-Wp} and it is guaranteed to be positive by properties of the Poisson distribution.

\begin{remark}\label{rem-poisson-weight}
We note that, we have exchanged the indicator weight  discussed in the previous section with what we shall call the {\em Poisson weight} $W_p$. This is the price to be paid for targeting (that is ensuring that $\tilde{Z}$ lands in a state consistent with the observation).  
However, the Poisson weight has lower variance than the indicator weight. See Appendix A, where we obtain a lower bound for the effective sample fraction (ESF) ($\E_0^2[W_p]/\E_0[W_p^2]$) of the Poisson weight in terms of the ESF of the indicator weight. 
\end{remark}

\begin{algorithm}[ht]  
\caption{Targeting Algorithm}
\label{alg-target} 
\begin{algorithmic}[1]
\State \textbf{Input:} Initial distribution $\mu(z_0)$,  $T$, filter sample size $N_s$, observation $y$, and matrices $A$, $B$ and $\nu$.
\State Make a judicious choice of piecewise-constant intensity function $\lambda(t,z_0,y)$ that is strictly positive.  
\For { $i=1$ to $N_s$ }
\State Generate $(\tilde{Z}_0^{(i)},K^{(i)})$ along with weight $W_p^{(i)}$ using Algorithm \ref{alg-endpoint}.
\State Generate trajectories $\tilde{Z}^{(i)}(t)$, $I^{(i)}(t)$ for $t \in [0,T]$, compute $\tilde{L}^{(i)}(T)$ and the path functional $C^{(i)}=\phi(Z^{(i)})$ using Algorithm \ref{alg-interpolation}.
\State $W^{(i)} \leftarrow W_p^{(i)} \tilde{L}^{(i)}(T)$.
\EndFor
\State {\bf Output:} Weighted sample $(C^{(i)}, Z^{(i)}, W^{(i)})$ for $i=1,\dots,N_s$. 
\end{algorithmic}
\end{algorithm}

\begin{algorithm}[h]  
\caption{Endpoint Algorithm}
\label{alg-endpoint} 
\begin{algorithmic}[1]
\State \textbf{Input:} Initial distribution $\mu(z_0)$,  $T$, observation $y$, piecewise constant intensity function $\lambda(t,z_0,y)$, matrices $A$ and $B$.
\State Generate $\tilde{Z}_0=(U_0,V_0)$ according to $\mu$. 
\State Compute $M'=\int_0^T \lambda'(t,\tilde{Z}_0,y)dt$. 
\State Generate $K'$ as a Poisson vector with mean $M'$ (independent components).
\State Set $K''=B^{-1}(y-V_0-A \, K')$. If $K'' \notin \posint^{m_2}$ go back to Step 2.
\State Set $W_p=\rho(K'',\tilde{Z}_0,y)$. 
\State {\bf Output:} $\tilde{Z}_0=(U_0,V_0)$, $K$ and $W_p>0$. 
\end{algorithmic}
\end{algorithm}

\begin{algorithm}[ht]  
\caption{Poisson Interpolation Algorithm} 
\label{alg-interpolation} 
\begin{algorithmic}[1]
\State \textbf{Input:} Initial state $z_0$, final time $T$, intervals $[t_{l-1},t_l]$ for $l=1,\dots,N$, total reaction count $k \in \posint^m$, intensity matrix $\overline{\lambda} \in \real^{m\times N}_+$ where $\overline{\lambda}_{jl} = \lambda_j(t)$ for  $t \in [t_{l-1}, t_l)$.
Define $\Delta t_l = t_l-t_{l-1}$. 
\State Initialize $t \leftarrow 0$, $I(t) \leftarrow 0$, $\tilde{Z}(0) \leftarrow z_0$, $\tilde{L}(T) \leftarrow 1$. 
\State Compute $\mu_j = \sum_{l=1}^{N} \overline{\lambda}_{j,l} \, \Delta t_l$ for $j=1,...,m$.
\For{$l = 1: N$}
\State Generate reaction count $r_{j,l}$ in $[t_{l-1}, t_l)$ 
$\sim$ Bino$(k_j, \frac{\overline{\lambda}_{j,l} \,\Delta t}{\mu_j})$ for $j = 1,... m$.
\State Compute $r_l = \sum_{j=1}^m r_{j,l}$
\State Generate jump times $t_j^i$ uniformly distributed in $[t_{l-1}, t_l)$ for $j=1,...,m$, $i = 1, ..., r_{j,l}$.
\State Sort $t_j^i$ so that $t_{(1)} < t_{(2)} ...< t_{(r_l)}$, where $t_{(p)}$ is the $p$th smallest among $t^j_i$, and record the corresponding reaction index $j_{(p)}$
for $p=1,...,r_l$.
\For{$p=1:r_l$}
\State $\tilde{L}(T) \leftarrow \tilde{L}(T) \times \exp \left( \sum_{j=1}^m (\overline{\lambda}_{j,l} - a_j(\tilde{Z}(t))) (t_{(p)}-t) \right)$
\State $t \leftarrow t_{(p)}, \;\; I_{j_{(p)}}(t) \leftarrow I_{j_{(p)}}(t) + 1, \;\; \tilde{Z}(t) \leftarrow \tilde{Z}(t) + \nu_{j_{(p)}}$.
\State $\tilde{L}(T) \leftarrow \tilde{L}(T) \times \frac{a_{j_{(p)}}(\tilde{Z}(t))}{\overline{\lambda}_{j^{(p)},l}}$
\EndFor
\State $\tilde{L}(T) \leftarrow \tilde{L}(T) \times \exp\left( \sum_{j=1}^m (\overline{\lambda}_{j,N} - a_j(\tilde{Z}(t))) (t_l-t) \right)$
\State $t \leftarrow t_l$.
\State $\mu_i \leftarrow \mu_i - \overline{\lambda}_{i,j} \, \Delta t$, \, $k_i \leftarrow k_i - r_{i,j}$ for $i=1,..,m$
\EndFor
\State {\bf Output:} Jump times $t \in [0,T]$, and $I(t)$, $\tilde{Z}(t)$ for jump times $t \in [0, T]$ and $\tilde{L}(T)$.
\end{algorithmic}
\end{algorithm}

\subsubsection{Second step of the Targeting Method} Having obtained ``end point'' proxies $\tilde{Z}_0$ and $K$ (along with weight $W_p$) for the initial state and the total reaction counts over $[0,T]$,  the second step involves generating the process $I(t)$ which is a proxy for $R(t)$ such that it
terminates at $K$, that is, $I(T)=K$. This second step is accomplished by the fact that the conditioned sample paths of an inhomogeneous Poisson process (conditioned on the value at a terminal time $T$) may be generated in a straightforward manner. Along with this interpolated processes $I(t)$, a Girsanov weight $\tilde{L}(T)$ is computed account for the change of measure between the proposed inhomogeneous intensity and the actual reaction propensities. The interpolation algorithm is described in Algorithm \ref{alg-interpolation}, which we explain below in detail.

Consider an inhomogeneous Poisson process $J(t)$ with a strictly positive intensity function 
$\alpha(t)$. By the random time change representation, it is possible to write 
$$J(t)=J_0\left(\int_0^t \alpha(s) ds\right),$$
where $J_0(t)$ is a unit rate Poisson process. Suppose we  condition on the event $J(T)=k$ (where $k \in \posint$). This is equivalent to conditioning on the event $$J_0\left(\int_0^T \alpha(s) ds\right)=k.$$ Then, the $k$ jump times of $J_0$ are i.i.d.\ uniform on the interval $[0, \int_0^T \alpha(s) ds]$. 
Thus, conditioned trajectories of $J(t)$ 
may be obtained as follows: 
\begin{itemize}
    \item Generate i.i.d.\ uniform random variables $u_1,\dots,u_k$ in $[0,1]$, sort them in increasing order and rename so that $$0 <u_1 < u_2 < \dots < u_k<1.$$ 
    \item Compute the jump times $0<v_1<v_2 < \dots <v_k<T$ of $J$ by 
    $$v_i = \eta^{-1}(u_i), \quad \text{ where } \eta(t)=\int_0^t \alpha(s) ds.$$
Since $\eta$ is strictly increasing it is invertible.
\end{itemize}

In our situation, the above strategy can be used to generate conditioned sample paths of the proxy processes $I_j(t)$ 
(conditioned on the initial state value $\tilde{Z}_0$ and on the final count value $I_j(T)=k_j$) by the use of the intensity function $\lambda_j(t,\tilde{Z}_0,y)$. The inverse function of $\eta_j(t) = \int_0^t \lambda_j(s,\tilde{Z}_0,y) ds$ is most tractable when $\lambda_j$ is piecewise constant in time. We have employed only piecewise constant intensity functions in our numerical examples and the  Poisson Interpolation Algorithm \ref{alg-interpolation} is also specialized for this case. 
However, other alternatives may be explored in the future.

For the implementation of a piecewise constant vector intensity function $\lambda(t,z_0,y)$, we consider a subdivision of $[0,T]$ into $N$ sub-intervals $[t_{l-1},t_l)$ and take $\lambda$ to be constant inside each sub-interval. Thus, the vector intensity function $\lambda$ can be described by a matrix $\overline{\lambda}_{j,l}$ for $j=1,\dots,m$ and $l=1,\dots,N$. We note that the number of jumps for the interpolated process in a sub-interval are binomial random variables. This is what we have implemented in our examples and is described in Algorithm \ref{alg-interpolation}. In Section \ref{sec-lambda-choice} we discuss in detail some possible ways to choose the piecewise constant intensity function $\lambda$. 

Since we used an inhomogeneous Poisson process for the reaction counts, we have to correct for this by an appropriate Girsanov change of measure. The change of measure is accounted for by computing a weight $\tilde{L}(T)$ which is the terminal time value of process $\tilde{L}(t)$ which can be computed along with the proxy processes $I(t)$ and $\tilde{Z}(t)$ \cite{bremaud}. 
First we note that $\tilde{Z}(t)=\tilde{Z}_0+\nu I(t)$. 
Define the processes $\tilde{L}_j:[0,T] \times \Omega \to \real$ for $j=1,\dots,m$
by
\begin{equation}
\tilde{L}_j(t) = \left(\prod_{i \geq 1} \frac{a_j(\tilde{Z}(\tilde{T}^j_i-))}{\lambda_j(\tilde{T}^j_i-,\tilde{Z}_0,y)} 1_{\{\tilde{T}^j_i \leq t\}}\right) \, \exp \left( \int_0^t (\lambda_j(s,\tilde{Z}_0,y) - a_j(\tilde{Z}(s)) \, ds \right),\\
\end{equation}
where $\tilde{T}^j_i$ for $i=1,2,\dots$ 
is the $i$th jump time of $I_j$. 
Then the process $\tilde{L}$ is given by
\begin{equation}
\tilde{L}(t) = \tilde{L}_1(t) \dots \tilde{L}_m(t).    
\end{equation}
A rigorous treatment  of the targeting algorithm is provided in Appendix B, where it is also shown  that if $P$ denotes the probability measure under which the lab process $(Z,R)$ has the correct propensities, then 
\[
P\{Z_0=z_0, R(t)=k \, | \, Y(T)=y\} = \frac{\E_0[ 1_{\{(z_0,k)\}}(\tilde{Z}_0,I(t)) \, W_p \, \tilde{L}(T) \, | Y(T)=y] }{\E_0[W_p \, \tilde{L}(T) \, | Y(T)=y]}.
\]




\subsection{Multiple Observation Snapshots}\label{sec-multiple}

In Section \ref{sec-targeting} we described a (weighted) Monte Carlo algorithm 
for the {\em targeting problem}. That is, to compute $\pi(t,z)$ for $0 \leq t \leq T$ in the context of a single snapshot observation at time $T>0$ given the observation $Y(T)=y$ and initial distribution $\mu$ for $Z_0$. This algorithm produces as output an i.i.d.\ weighted sample 
$(C^{(i)},W^{(i)})$ of size $N_s$. Here, $C^{(i)}$ is a path functional which could contain for instance the values of the proxy $\tilde{Z}$ at certain time points $T_i$ in $[0,T]$ including $T$. 
In particular, the targeting algorithm over $[0, T]$ can be used to obtain an empirical distribution 
\[
\hat{\mu} = \frac{\sum_{j=1}^{N_s} W^{(i)} \, \delta_{\tilde{Z}^{(i)}(T)}}{\sum_{j=1}^{N_s} W^{(i)}},
\]
for the conditional distribution of $Z(T)$ given $Y(T)$. Here $\delta_{z}$ is the Dirac measure concentrated at $z$.

Suppose we have multiple observation snapshots: $Y(t_l)=y_l$ for $l=1,\dots,T_s$ with $0 \leq t_1 < t_2 < \dots t_{T_s}$. 
We may handle this situation by repeatedly applying the targeting algorithm between successive snapshots $t_{l-1}$ and $t_l$ as described below.

For the first span, if $t_1>0$, we set $l=1$, let $t_{l-1}=t_0=0$. Otherwise 
$t_1=0$, so we set $l=2$. Then we apply the targeting algorithm for the span $[t_{l-1},t_l]$. 
We take the initial probability measure $\mu$ as the known probability distribution of $Z_0$. If $t_1=0$, the known observation $Y(0)=y_0$ is incorporated into the knowledge of $\mu$.  

The targeting algorithm applied to a general span $[t_{l-1},t_l]$ can be used  
to obtain a weighted sample $(Z^{(i)}(t_l),W^{(i)}(t_l))$ which gives us an empirical measure $\hat{\mu}_l$ defined by
\[
\hat{\mu}_l = \frac{\sum_{j=1}^{N_s} W^{(i)}(t_l) \, \delta_{\tilde{Z}^{(i)}(t_l)}}{\sum_{j=1}^{N_s} W^{(i)}(t_l)}.
\]
We may use $\hat{\mu}_l$ as an approximation of the conditional  
probability distribution of $Z(t_l)$ given the 
observations $\{Y(t_1)=y_1, \dots, Y(t_l)=y_l\}$. This allows us to proceed to the next interval $[t_l, t_{l+1}]$ where we ideally need the knowledge of the true conditional distribution of $Z(t_l)$ given the observations $Y(t_1)=y_1,\dots,Y(t_l)=y_l$. However, we simply use $\hat{\mu}_l$ in its place as an approximation.

\section{Examples} \label{sec-examples}

In this section, we briefly describe the reaction systems that we shall use as numerical examples. The first two are monomolecular and hence have linear propensities while the third example has a bimolecular reaction making it nonlinear.
These examples are chosen so that it is possible to compute the conditional probability of interest either analytically or via deterministic computations such as ODE solvers. This way, we can verify the accuracy of our Monte Carlo methods. 
We describe how the conditional probabilities may be computed for each example. We also note that the exact conditional propensity may be computed if the transition probabilities 
$P\{Z(T)=z_T \, | Z(t)=z_t\}$ ($0 \leq t \leq T$) can be computed.  
%




\subsection{Example 1: Pure-death model} \label{sec-eg-1}

\begin{equation}
   S \stackrel{c}{\longrightarrow} \varnothing
\end{equation}
with the propensity $a(x) = c x$. This is the simplest example where an analytically tractable conditional propensity in addition to conditional probabilities is available. 

The conditional distribution
\begin{equation}
    \pi(x_t | x_0, x_T) := P\{X(t) = x_t \,|\, X(0) = x_0, X(T)=x_T\},
\end{equation}
is given in terms of the transition probabilities as
\begin{equation} \label{eq-cond-transition}
\begin{split}
    &\pi(x_t | x_0, x_T) 
    = \frac{P\{X(T) = x_T \,|\, X(t)=x_t\} \, P\{X(t) = x_t \,|\, X(0) = x_0\}}{P\{X(T) = x_T \,|\, X(0) = x_0\}}.
\end{split}    
\end{equation}
where transition probability from $t_1$ to $t_2$ $(t_1 \leq t_2)$ follows a binomial distribution with parameters $n=x_1$ and $p=e^{-c(t_2-t_1)}$:
\begin{equation} \label{eq-trans-bino}
    P\{X(t_2) = x_2 \,|\, X(t_1) = x_1\} = \binom{x_1}{x_2} e^{-c(t_2-t_1) x_2} (1-e^{-c(t_2-t_1)})^{x_1-x_2}.
\end{equation}

Using equations \eqref{eq-cond-transition} and \eqref{eq-trans-bino}, we can obtain that
\begin{equation}
\begin{split}
    \pi(x_t | x_0, x_T) = \binom{x_0-x_T}{x_0 - x_t} \left(\frac{e^{-ct}-e^{-cT}}{1-e^{-cT}}\right)^{x_t-x_T}\left(\frac{1-e^{-ct}}{1-e^{-cT}}\right)^{x_0-x_t}.
\end{split}
\end{equation}
Furthermore, we can also obtain the conditioned propensity as
\begin{equation} \label{eq-cond-prop-decay}
    \lambda(x_t \,|\, x_T) = a(x_t) \frac{p(x_T | X(t)=x_t-1)}{p(x_T | X(t)=x_t)}
    = \frac{c(x_t-x_T)}{1-e^{-c(T-t)}}.
\end{equation}
We note that the conditional propensity depends explicitly on time $t$ and hence the waiting times between reaction events according to the conditional propensity are not exponential. 
Yet, the waiting time has an analytically tractable form. 
Let us denote the wait time at time $t$ for the next reaction to occur by $\tau$ and let 
\begin{equation}
G(s)=P\{\tau>s \, | X(t)=x_t, X(T)=x_T\}. 
\end{equation}
Then it can be shown that
\begin{equation}
    G(s) = \left(\frac{e^{-cs}-e^{-c(T-t)}}{1-e^{-c(T-t)}}\right)^{x_t-x_T}.
\end{equation}
For Monte Carlo simulation, a sample for $\tau$ can be 
obtained by 

\begin{equation}
\tau = G^{-1}(u)
\end{equation}
where $u$ is a uniform random variable in $[0,1]$ and $G^{-1}$ is the inverse of the strictly monotone function $G$.


\subsection{Example 2: Reversible Reaction $S_1 \longleftrightarrow S_2$}
We consider the system 
\begin{equation}
    S_1 \stackrel{c_1}{\longrightarrow} S_2, \quad S_2 \stackrel{c_2}{\longrightarrow} S_1.
\end{equation}
Let $Z(t) = (\#S_1(t), \#S_2(t))$. Suppose the propensities take mass-action form $a_1(z)=c_1 z_1$ and $a_2(z)=c_2 z_2$.  Denote the initial state $z_0 = (z_{01}, z_{02} )$.

To find the evolution of the probability distribution, we consider an associated mono-molecular system \cite{rathinam2007reversible} which is described by a random walk $M(t)$ between two states. This is a continuous-time Markov chain with the rate matrix
\begin{equation}
    Q = \begin{pmatrix}
    -c_1 & c_1 \\
    c_2 & -c_2
    \end{pmatrix}.
\end{equation}
We can obtain its transition probability matrix $P(t)$ by Kolmogorov's backward equation $P'(t) = Q\,P(t)$, where $P_{i,j}(t) = $ Prob$\{M(t)=j \,|\, M(0)=i\}$ for $i,j = 1, 2$. 
We note the conservation relation $Z_1(t)+Z_2(t)=z_{01}+z_{02}$. It follows that
\begin{equation} \label{eq-trans-eg2}
\begin{split}
    &P\{Z_1(t)=z_t, Z_2(t)=z_{01}+z_{02}-z_t \,|\, Z_1(0) = z_{01}, Z_2(0)= z_{02}\}\\
    &= \sum_{k=0}^{z_{01}+z_{02}}\text{Binopdf}(k; z_{01}, P_{11}(t))\text{Binopdf}(z_t-z_{01}-k; z_{02}, P_{21}(t)).
\end{split}
\end{equation}
Using equations \eqref{eq-trans-eg2} and \eqref{eq-cond-transition}, we could obtain an analytical formula for the conditional distributions.
We also note that since the transition probability can be computed exactly from \eqref{eq-trans-eg2}, it is possible to compute the conditional propensities exactly. However, unlike in Example 1, here the wait times according to the conditional propensities are not analytically tractable. Nevertheless, one may use an approximate conditional propensity which is held constant in between jumps as in \cite{golightly2019efficient}.  

\subsection{Example 3: $S_1 \longleftrightarrow S_2, \quad S_1+S_2 \longleftrightarrow S_3$}
We consider
\begin{equation}
    S_1 \stackrel{c_1}{\longrightarrow} S_2, \quad S_2 \stackrel{c_2}{\longrightarrow} S_1, \quad S_1 + S_2 \stackrel{c_3}{\longrightarrow} S_3, \quad S_3 \stackrel{c_4}{\longrightarrow} S_1 + S_2.
\end{equation}
Let $Z(t) = (\#S_1(t), \#S_2(t), \#S_3(t))$. Suppose we observe species $S_3$, so $X(t)=(\#S_1(t), \#S_2(t))$ and $Y(t)=\#S_3(t)$. Suppose $a_1(z)=c_1 z_1$, $a_2(z)=c_2 z_2$, $a_3(z) = c_3 z_1 z_2$ and $a_4(z)=c_4 z_3$.  


The conditional distribution $\pi(t,z)$ is easy to obtain if we take $t=T$. Let $Z_0 = (z_{10}, z_{20}, z_{30})$, due to the conservation relation $Z_1(t)+Z_2(t)+ 2 Z_3(t) = z_{10}+z_{20}+2z_{30}$, the state space is finite and thus 
 we can solve Kolmogorov's forward equation numerically via an ODE solver to obtain the unconditional p.m.f.\ $p(T,z)$. 
Given $Y(T)=\#S_3(T)=y$, and $p(T,z)$ for the joint distribution of $(Z_1, Z_2, Z_3)$, we can easily compute the conditional distribution $\pi(T,z)$. 

\section{Numerical experiments} \label{sec-eg-discrete-time}

We shall use the three examples described in Section \ref{sec-examples} to investigate the effectiveness of the various Monte Carlo algorithms.  
We shall focus on the targeting problem. That is, we know the initial distribution $\mu_0$ which we usually take to be Dirac and we have a partial state observation $Y(T)=y$ at time $T>0$.


We use the mean total variation error (TVE) to characterize the error in the estimated conditional distribution $\hat{\pi}(t,z)$: 
\begin{equation*}
    \text{TVE} = \bE\left[ \sum_z |\hat{\pi}(t,z) - \pi(t,z)| \,\Big|\, Y(T)=y\right].
\end{equation*}
We also record the effective sample fraction (ESF) defined by the ratio of effective sample size $N_e$ over the filter sample size $N_s$
\[
    \text{ESF} = \frac{N_e}{N_s} = \frac{1}{N_s} \frac{\left(\sum_{i=1}^{N_s} W^i\right)^2}{\sum_{i=1}^{N_s} (W^i)^2}.
\]
The subsections are organized so that each subsection explores a certain choice made within our targeting algorithm or compares our targeting algorithm with other methods. 

\subsection{Choice of inhomogeneous Poisson intensities $\lambda$}\label{sec-lambda-choice}
Our targeting method involves choosing inhomogeneous Poisson intensities $\lambda$ for $R(t)$. 
It is most convenient to interpolate process $I$ when $\lambda$ is piecewise constant in time. We evenly divide $[0, T]$ into $N$ subintervals $[t_j,t_{j+1}]$ with $t_{j+1}-t_j =\Delta t = \frac{T}{N}$ and let $\lambda(t) = \lambda(t_j)$ for $t \in [t_j, t_{j+1})$. We use a matrix representation for this piecewise constant intensity $\overline{\lambda} \in \real^{m\times N}_+$, where $\overline{\lambda}_{i,j} = \lambda_i(t_j)$.
Here we present two strategies for choosing the piecewise constant intensities that we used in our numerical experiments.
Since the efficiency of the Monte Carlo method depends on having a high effective sample size, we ideally want the weight $W \tilde{L}(T)$ to have low variance. Recall that $W$ is the ``Poisson weight'' that appears in the early steps of the targeting method (see Algorithm \ref{alg-target}) while $\tilde{L}(T)$ is the weight due to the Girsanov transformation. 

The first idea is to focus only on the Girsanov weight $\tilde{L}(T)$ since this appears to be more variable in our numerical experiments. Intuitively, the closer $\lambda(t)$ is to the propensity $a(Z(t))$, the better. Since $\lambda(t)$ is state independent, a natural choice is $\bE[a(Z(t))]$. If we know $\bE[a(Z(t))]$ analytically, we can construct a left endpoint approximation $\overline{\lambda}_1 (t) = \bE[a(Z(t_j))]$ for $t \in [t_j, t_{j+1)})$. 
If $\bE[a(Z(t))]$ is not analytically tractable, we could use some reasonable approximation of it. 
Examples include using the solution of the corresponding (deterministic) reaction rate equations as an approximation or, alternatively, estimate $\bE[a(Z(t))]$ via an independent forward Monte Carlo simulation. 

The above idea doesn't take the observation $y$ into account and we note that the variability in the Poisson weight $W$ is dependent on $y$. 
Thus an alternative choice for the piecewise constant Poisson intensity $\overline{\lambda}_2$ is to improve the probability of the observed value $y$. So we impose the condition that the expectation of the process (as predicted by the intensity $\overline{\lambda}_2$ lands precisely at the desired state: $\nu''\int_0^T \overline{\lambda}_2(t) dt = y-y_0$, 
while keeping $\overline{\lambda}_2$ close to the average intensity $\overline{\lambda}_1$. 
To be specific, we solve the following constrained optimization problem to get intensity $\overline{\lambda}_2$:
\begin{equation} \label{eqn-propensity-optimization}
\begin{split}
    \overline{\lambda}_2 = \text{argmin}_{\lambda \in \real^{m\times N}_+} \| \lambda - \overline{\lambda}_1\|_F
    \quad  \text{ s.t. } \nu'' r = \Delta y, \text{ where } r_i = \sum_{j=1}^{N} \lambda_{i,j} \Delta t.
\end{split}
\end{equation}

Since we need the inhomogeneous Poisson intensity to be strictly positive, in order to avoid zero intensity, 
if the solution to \eqref{eqn-propensity-optimization} does not return strictly positive $\lambda_{ij}$ 
we use the optimization problem \eqref{eqn-propensity-optimization-2} instead. 
The set $S$ in \eqref{eqn-propensity-optimization-2} is defined as follows. 
Let $\underline{\lambda}_i = \min\{a_i(z) \,|\, a_i(z)>0, z \in \posint^n\}$ for $i=1,\dots,m$.  
Set $S = \{\lambda \in \real^{m\times N}_+ \,|\, \lambda_{i,j} \geq \underline{\lambda}_i, \text{for \,} j=1, ..., N\}$. 
\begin{equation} \label{eqn-propensity-optimization-2}
\begin{split}
    \overline{\lambda}_2 = \text{argmin}_{\lambda \in S} \| \lambda - \overline{\lambda}_1\|_F
    \quad  \text{ s.t. } \nu'' r = \Delta y, \text{ where } r_i = \sum_{j=1}^{N} \lambda_{i,j} \Delta t.
\end{split}
\end{equation}
Here $\| \cdot \|_F$ stands for the Frobenius norm. In our numerical examples, we use MATLAB program \texttt{fmincon} to solve this optimization problem. 



\begin{table}[htbp] 
    \centering
    \begin{adjustbox}{width=1.2\textwidth,center=\textwidth}
    \begin{tabular}{|c| c | c| c | c | c| c | c | c | }
        \hline
        Method & Observation &TVE ($95\%$ interval) & ESF ($W)$ & ESF ($\tilde{L}$) &  ESF & Runtime \\
        \hline
        Naive &  $y=4$  & $0.1188, [0.1104, 0.1272]$ & N/A & N/A & $0.25$ & $0.05$\\
        Targeting $(\overline{\lambda}_1)$ & (common observation)  & $0.0722, [0.0673, 0.0771] $  & $0.95$ & $0.56$ & $0.66$ & $0.3$\\
        Targeting $(\overline{\lambda}_2)$ & $P(Y(T)=4)=0.245$  & $0.0760, [0.0702, 0.0818]$  &  $0.95$ & $0.51$ &  $0.61$  & $0.3$ \\
        CP (approx) &   & $0.0834, [0.0783, 0.0886] $ & N/A & N/A & $0.52 $ & $31.3$\\
        \hline
        Naive & $y = 7$   & $0.3646, [0.3393, 0.3900] $  & N/A & N/A & $0.03 $ & $0.05$ \\
        Targeting $(\overline{\lambda}_1)$& (rare observation) & $0.0940, [0.0880, 0.1001]$ &  $0.76$ & $0.37$ & $0.38 $ & $0.3$\\
        Targeting $(\overline{\lambda}_2)$& $P(Y(T)=7)=0.027$ & $0.0962, [0.0897, 0.1028]$ & $0.98$ & $0.21$ &  $0.37$ & $0.3$\\
        CP (approx) &  & $0.0907, [0.0851, 0.0962] $ & N/A & N/A & $0.43$ & $33.3$\\
        \hline
    \end{tabular}
    \end{adjustbox}
    \caption{Example 2. Reversible isomerization reaction $S_1 \longleftrightarrow S_2$.  We took $z_0 = (10, 0)$, $c=(1, 1.5)$, $T=1$, $t=0.7$, $\Delta t = 0.1$, while taking a common observation as well as a rare observation. We used $N_r = 100$ trials, while each trial used filter sample size $N_s = 1000$.}
    \label{tab-eg2-small}
\end{table}

The difference between the two choices $\overline{\lambda}_1$ and $\overline{\lambda}_2$ are illustrated by Examples 2, 3 and 1. 
The naive approach is also provided for comparison. In all examples, we choose a ``common observation'' for $y$, meaning a value of $y$  close to the mode of marginal distribution for $Y(T)$ and a ``rare observation'' for $y$, meaning a value of $y$ that has a significantly smaller probability compared to the common observation. 

In Example 2, $S_1 \longleftrightarrow S_2$, 
we picked $I_1(T)$ (corresponding to $S_1 \longrightarrow S_2$) as free reaction set $I_2 = I_1 - \Delta y$.
we first took a smaller system (small total copy number): $Z(0) = (10, 0)$. We chose parameters $c = (1, 1.5)$ and also $t=0.7$, $T=1$, $\Delta t = 0.1$. The results are shown in Table \ref{tab-eg2-small}. We also tried with a larger system with initial state $Z(0)=(100, 100)$. The results are shown Table \ref{tab-eg2-large-sys}.  
In Example 3, we took $X(0)=(20, 20, 20)$, $c=(0.5, 1, 0.1, 1)$, $t=1$, $T=1$. The results are shown in Table \ref{tab-three-species}.

\begin{table}[htbp] 
    \centering
    \begin{adjustbox}{width=1.2\textwidth,center=\textwidth}
    \begin{tabular}{|c| c | c| c | c | c|  c |c | c |}
        \hline
        Method & Observation &TVE  & ESF ($W)$ & ESF ($\tilde{L}$) & ESF &Runtime\\
         &  &($95\%$ interval) &   &  & & (seconds)\\
        \hline
        Naive &  $y=80$  & $0.1868, [0.1826, 0.1910]$& N/A & N/A  & $0.06$ & $19.50$\\
        Targeting $(\overline{\lambda}_1)$ & (common observation)  & $0.0665, [0.0646, 0.0684] $ & $0.86$ &  $0.51$ & $0.43 $ & $21.10$\\
        Targeting $(\overline{\lambda}_2)$ & $(p=0.0575)$ & $0.0626, [0.0607, 0.0646] $ & $0.88$ & $0.50$ & $0.44$ & $21.93$\\
        \hline
        Naive & $y = 98$   & $0.6764, [0.6595, 0.6933] $ & N/A & N/A & $0.004$ & $19.03$\\
        Targeting $(\overline{\lambda}_1)$& (rare observation) & $0.0928, [0.0902, 0.0954] $ & $0.62$ & $0.35$ & $0.21$ & $21.30$\\
        Targeting $(\overline{\lambda}_2)$& $(p = 0.0043)$  & $0.0755, [0.0733, 0.0776]$ & $0.87$ &  $0.35$ & $0.31$ & $21.95$\\
        \hline
    \end{tabular}
    \end{adjustbox}
    \caption{Example 2. Reversible isomerization reaction $S_1 \longleftrightarrow S_2$. We took $z_0 = (100, 100)$, $c=(1, 1.5)$, $T=2$, $t=0.7$, $\Delta t = 0.25$, while taking a common observation $y_T = 80$ as well as a rare observation $y_T =  98$. We used $N_r = 100$ trials, while each trial used filter sample size $N_s = 10,000$.}
    \label{tab-eg2-large-sys}
\end{table}

\begin{table}[htbp]
    \centering
    \begin{adjustbox}{width=\textwidth,center=\textwidth}
    \begin{tabular}{|c| c|  c | c| c | c | c | c | c| }
        \hline
        Method & Observation  &TVE & ESF $(W)$  & ESF $(\tilde{L})$ & ESF & Runtime\\
          & & ($95\%$ interval) &  & & &(seconds)\\
        \hline
        Naive  & common observation & $0.2146, [0.2117, 0.2174]$ & N/A & N/A & $0.16$ & $0.331$\\
        Targeting, $\overline{\lambda}_1$ &   $y = 24$ & $0.1695, [0.1669, 0.1722]$ & $0.83$ & $0.20$ & $0.21$ & $0.773$\\
        Targeting, $\overline{\lambda}_2$ & &  $0.1699, [0.1673, 0.1726] $ &  $0.84$ & $0.20$ & $0.21$ & $0.770$\\
        \hline
         Naive  & rare observation & $0.4347, [0.4286, 0.4408] $ & N/A & N/A & $0.04$ & $0.326$\\
        Targeting, $\overline{\lambda}_1$ & $y = 20$ & $0.2184, [0.2152, 0.2217]$ &  $0.70$ & $0.12$ & $0.14$ & $0.791$\\
        Targeting, $\overline{\lambda}_2$ & &  $0.2103, [0.2071, 0.2135]$ & $0.84$ &
        $0.14$ & $0.15$ & $0.794$\\
        \hline
    \end{tabular}
    \end{adjustbox}
    \caption{Example 3. $S_1 \longleftrightarrow S_2, S_1 + S_2 \longleftrightarrow S_3$. We took  $Z(0)=(20,20,20)$, $c=(0.5, 1, 0.1, 1)$, $t = T = 1$, $\Delta t = 0.1$. We used $N_r = 1,000$ trials with filter sample size $N_s = 1,000$.}
    \label{tab-three-species}
\end{table}

The intensity choices $\overline{\lambda}_1$ and $\overline{\lambda}_2$ have similar performance in the common observation scenarios. In the rare observation scenarios, $\overline{\lambda}_2$ tends to have better performance. We also note that the optimization often gives a better effective sample fraction (ESF) for the Poisson weights, especially when the observation is rare. We also observe that 
bye and large $\overline{\lambda}_1$ has similar or better ESF for the Girsanov weights when compared to $\overline{\lambda}_2$. This is expected since the choice of $\overline{\lambda}_1$ is solely focused on reducing the variance of the Girsanov weight. 
It is also notable that $\overline{\lambda}_2$ has better ESF for the Poisson weights especially when rare observations are concerned. This is not surprising since the choice of $\overline{\lambda}_2$ is based on a steering towards the observation $y$ to improve the ESF of the Poisson weights. The overall ESF is also better for $\overline{\lambda}_2$ in the case of rare observations (with one exception), assuring that the constrained optimization is having the desired effect.  

Finally, in Example 1, we took $x_0 = 1000$, $c=4$, $T=0.5$, $t=0.2$, and time mesh $\Delta t = 0.02$. See Table \ref{tab-eg1}. In this example, the free dimension $m_1=0$, while the determined dimension $m_2=1$, so the Poisson weight will always be the same for different samples making ESF equal to $1$ for the Poisson weights. This explains why $\overline{\lambda}_1$ does better than $\overline{\lambda}_2$ because only the Girsanov weights matter. 

\begin{table}[htbp]
    \centering
    \begin{adjustbox}{width=1.2\textwidth,center=\textwidth}
    \begin{tabular}{|c| c | c| c | c | c|  c | }
        \hline
        Method & Observation & Total Variation Error  & Effective Sample  & Runtime\\
        &  &  ($95\%$ interval)  & Fraction (ESF) & (seconds)\\
        \hline
        Naive & $x_T = 368$   & $1.1353, [1.1077, 1.1630] $ & $0.027 $ & $2.3$\\
        Targeting $(\overline{\lambda}_1)$ & (common observation)  & $0.2037, [0.1995, 0.2080]$  & $0.919 $ & $2.6$\\
        Targeting $(\overline{\lambda}_2)$&  & $0.2072, [0.2029, 0.2116]$ &   $0.900$ & $2.7$\\
        CP (approx) & & $0.3485, [0.3414, 0.3556]  $ & $0.322 $ & $2.4$\\
        CP (exact) &    & $0.1957, [0.1920, 0.1994] $ &  $1.000$ & $2.4$\\
        \hline

        Naive & $x_T = 404$   & $1.9036, [1.8934, 1.9139]$   &  $0.002$ & $2.4$\\
        Targeting $(\overline{\lambda}_1)$& (rare observation)  & $0.1979, [0.1942, 0.2016] $ &  $0.923$ & $2.6$\\
        Targeting $(\overline{\lambda}_2)$& & $0.2297, [0.2248, 0.2347] $  & $0.654$ & $2.4$\\
        CP (approx) & & $0.3351, [0.3275, 0.3427] $  & $0.324 $ & $2.1$\\
        CP (exact) &    & $0.1909, [0.1872, 0.1947] $  &  $1.000$ & $2.3$\\
        \hline
    \end{tabular}
    \end{adjustbox}
    \caption{Example 1. We took $x_0 = 1000$, $c=2$, $T=0.5$, $t=0.2$, $\Delta t = 0.02$.
    We used $N_r = 100$ trials, while each trial has filter sample size $N_s = 1000$. We note that in our experiment for the rare observation case, the naive method has only $80\%$ of the trials provided meaningful results, and the rest of $20\%$ runs of the naive filter failed when none of the landed at the target state, resulting a zero denominator issue.}
    \label{tab-eg1}
\end{table}

\begin{figure}
    \centering
    \includegraphics[width=6cm]{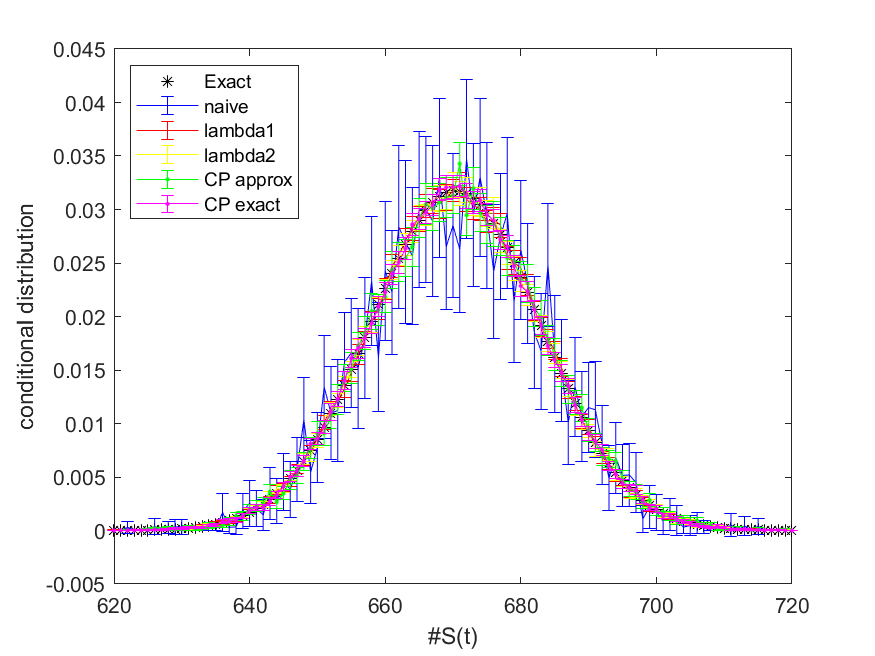}
    \includegraphics[width=6cm]{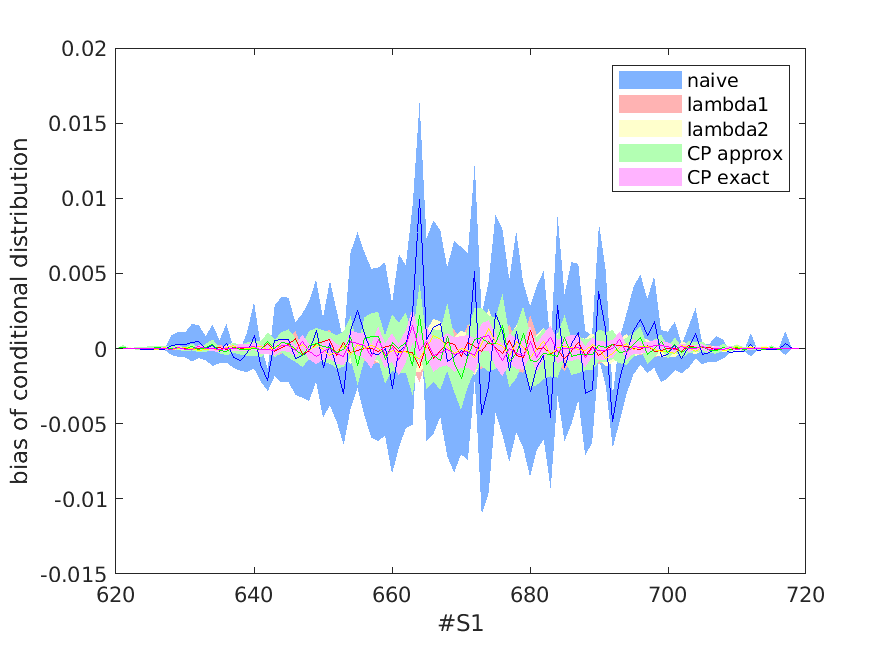}
    \caption{Example 1. Pure death example. Estimated conditional distribution compared with the actual conditional distribution. A common observation. $x_0 = 1000$, $c=2$, $T=0.5$, $t=0.2$. The estimate came with $95\%$ confident interval, which was estimated based on $N_r = 100$ trials, while each trial used filter sample size $N_s = 1000$.}
    \label{fig-cond-distr-rev2-rare}
\end{figure}

\subsection{Computational time and comparisons with the naive method}
For convenience, within each numerical experiment as depicted in the various tables, we ran the different methods with the same filter sample size. In order to make a fair comparison based on run time, one has to scale the errors by the square root of the ratio of run times or alternatively scale up the effective sample sizes ($\text{ESF} \times N_s$) by the ratio of run times.

When the total species counts are small, the naive method is noticeably faster (by a factor of up to 6) than the targeting methods while incurring greater errors for the same filter sample size. See Tables \ref{tab-eg2-small} and \ref{tab-three-species} for instance. To make a detailed comparison, consider Table \ref{tab-eg2-small} as an example. Here, the naive method runs about 6 times faster and hence one has to scale the mean TVE errors by $1/\sqrt{6} \approx 0.4$. This makes the naive method's performance better for the common observation case whereas for the rare observation case its performance is inferior. Alternatively, if we scale the effective sample sizes  of the naive method by a factor of $6$, we reach the same conclusion.

 However, when the total copy number of the species increases, the run time differences between the naive method and the targeting methods appear to shrink (see Tables \ref{tab-eg2-large-sys} and \ref{tab-stat-distance-three-species-common-100} for instance) while the effective sample fractions of the naive method drops as expected (see Remark \ref{rem-naive-ess}). 
 We believe that one reason for the relatively slower run times of the targeting methods is due to the overhead involved in the binomial random numbers generated inside the sub-intervals and this overhead becomes negligible as the copy numbers increase because the number of reaction events within each sub-interval increases with an increase in copy numbers. 
 We also like to add that we  have limited ourselves to systems where it is feasible to compute the exact conditional distribution for purposes of comparing the errors (mean TVEs). If we consider multi-dimensional observations and modestly large copy numbers, say in the order of 100s, we expect the performance of the naive method to be very poor based on the above considerations. On the other hand, the targeting methods seem to be more robust when copy numbers are scaled or the observations are somewhat rare.

Another factor to be considered is the length of the time span $T$. When $T$ increases, the Girsanov weight $\tilde{L}(T)$ is generally expected to show greater variance and thereby significantly limiting the performance of the targeting methods. There is however, a 
remedy, which involves resampling. This is discussed in Section \ref{sec-resample} and Table \ref{tab-eg2-resample} shows that if resampling is not applied, the naive method is better than the targeting method at least in the case of a common observation even when the copy numbers are in the order of 100s. However, with resampling, the targeting method performs better than the naive method. In the case of the naive method, resampling is not helpful within the span $[0,T]$ (more generally between two successive observations) because the weight degeneracy of the naive method only occurs at the final time $T$ (more generally at the observation times). This is because the weights remain $1$ until time $T$ is reached and then may become $0$.

Yet another factor to be considered is the overhead in setting up the intensity matrix which is an added cost. This computational cost however is independent of the filter sample size and hence we have not reported it.

Thus, on the whole, we may surmise that the targeting methods (with resampling if necessary)  perform better than the naive method when the copy numbers increase and also when the observations become rarer.

\subsection{Choice of free reaction channels}
When we solve linear equations $y=Y_0 + \nu'' I(T)$ to ensure $Y(T)=y$, we need to choose a set of reactions to be free, which determines the total reaction count of the rest of the reactions. There are different options. For instance, in Example 2, $S_1 \longleftrightarrow S_2$, to solve $y-Y_0 = I_1(T) - I_2(T)$, we could choose $I_1$ to be the free reaction, then set $I_2(T) = I_1(T)-y+Y_0$. Alternatively, if we choose $I_2$ to be free, then $I_1(T) = I_2(T)-y-Y_0$. We show the results in Table \ref{tab-free-slaved} and we see that there is a noticeable difference in the performances (we use $\overline{\lambda}_2$) of these two choices. This is a topic that may be worth exploring in the future. In the case of Example 2, this can be explained by some analysis which approximates Poissons by Gaussians. In the interest of space, we do not show it here.  

\begin{table}[htbp]
    \centering
    \begin{adjustbox}{width=1.2\textwidth,center=\textwidth}
    \begin{tabular}{|c| c | c| c | c | c| c | c | }
        \hline
        Method & Observation &TVE ($95\%$ interval) & ESF ($W)$ & ESF ($\tilde{L}$) &  ESF \\
        \hline
        Targeting $(\overline{\lambda}_2)$, $I_1$ free & common observation  & $0.1014, [0.0942, 0.1086]$  &  $0.77$ & $0.22$ &  $0.31$   \\
        Targeting $(\overline{\lambda}_2)$, $I_2$ free & $P(Y(T)=4)=0.245$ & $0.0978, [0.0887, 0.1070]$  &  $0.95$ & $0.32$ &  $0.39$  \\
        \hline
        Targeting $(\overline{\lambda}_2)$, $I_1$ free & rare observation  & $0.0940, [0.0880, 0.1001]$ & $0.76$ & $0.38$ &  $0.38$ \\
        Targeting $(\overline{\lambda}_2)$, $I_2$ free & $P(Y(T)=7)=0.027$ & $0.0921, [0.0864, 0.0977]$ & $0.98$ & $0.24$ &  $0.37$ \\
        \hline
    \end{tabular}
    \end{adjustbox}
    \caption{Example 2. Reversible Reaction $S_1 \longleftrightarrow S_2$. We took $z_0 = (10, 0)$, $c=(1, 1.5)$, $T=1$, $t=0.7$, $\Delta t = 0.1$, while taking a common observation as well as a rare observation. We used $N_r = 100$ trials, while each trial used filter sample size $N_s = 1000$.}
    \label{tab-free-slaved}
\end{table}

\subsection{Two-stage targeting algorithm}
We introduce the \textit{two-stage method} as another attempt to address the issue of growing variance of the weights. For estimating the system at time $t$ ($0 < t < T$), especially when $t$ is close to $T$, we could split the interval $[0, T]$ into two subintervals $[0, t_0]$ and $[t_0, T]$ where $0 \leq t_0\leq t \leq T$ and run the (unconditional) forward simulation in time interval $[0, t_0]$ and then apply the targeting algorithm in the interval $[t_0, T]$. This way, we reduce the length of time interval where the weights grow disparately. We also explore the choice of $t_0$ for splitting the interval.

We use Example 3 
\begin{equation}
    S_1 \stackrel{c_1}{\longrightarrow} S_2, \quad S_2 \stackrel{c_2}{\longrightarrow} S_1, \quad S_1 + S_2 \stackrel{c_3}{\longrightarrow} S_3, \quad S_3 \stackrel{c_4}{\longrightarrow} S_1 + S_2,
\end{equation}
where species $S_3$ is observed, to illustrate the two-stage algorithm and compare it with just using a single stage.

We took $Z(0)=(20,20,20)$, $c=(0.5, 1, 0.1, 1)$, $t=T=1$, and applied the naive method and several variants of targeting algorithm to estimate the conditional distribution of 
$Z_1(T)$. For an additional experiment with larger copy numbers, we took $Z(0)=(100,100,100)$, $c=(0.5, 1, 0.02, 1)$ (note the scaling of the reaction constant to account for scaling the volume in a bimolecular reaction) and kept $T=1$ (unchanged). We picked first three reactions as ``free reactions" and the last one is determined by $I_4 = I_3 - \Delta y$.

For the single stage, i.e.\ targeting over the interval $[0,T]$, we took mesh size $\Delta t = 0.1$.
Given the nonlinearity of the system, we used the deterministic system to determine $\lambda(t)$. That is, we solved the reaction rate equations (MATLAB ODE solver) 
\begin{equation*}
\begin{split}
    &\frac{dz_1}{dt} = -c_1 z_1 + c_2 z_1 - c_3 z_1 z_2 + c_4 z_3, 
    \quad \frac{dz_2}{dt} = c_1 z_1 - c_2 z_1 - c_3 z_1 z_2 + c_4 z_3, \\
    &\frac{dz_3}{dt} = c_3 z_1 z_2 - c_4 z_3,
\end{split}
\end{equation*}
to obtain (piecewise constant) $\overline{\lambda}_1$ and then $\overline{\lambda}_2$ as described earlier. 

For the two-stage targeting algorithm, since the targeting is only applied in the second stage $[t_0,T]$ which is intended 
to be a small interval, we simply used a constant value for the intensity. We explored two possibilities for the choice of 
this constant vector intensity.
One option we took was to use the sample mean of the propensity from the end of the first stage: $\overline{\lambda} = \frac{1}{N_s}\sum a(\tilde{Z}^{(i)}(t_0))$. The other option is to assign each individual sample $i$, a propensity that takes account of the state $\tilde{Z}^{(i)}(t_0)$. 
We chose $\lambda^{(i)} \in \real^{4}_{+}$ to be $\lambda^{(i)} = \max\{a(\tilde{Z}^{(i)}(t_0), a(1,1,1)\}$. The reason we introduce $a(1,1,1)$ is that $a(\tilde{Z}^{(i)}(t_1))$ could possibly contain zero components, while our method requires strictly positive intensities, so we enforce a small lower bound corresponding to the situation where there is one molecule for each species present. Our choice $a(1,1,1)$ is usually a small quantity compared to $a(\tilde{Z}^{(i)} (t_0)$ while sufficiently safeguarding the strict positivity of intensity $\lambda(t)$. 
Numerical results are summarized in Table \ref{tab-stat-distance-three-species-common} for a common observation and Table \ref{tab-stat-distance-three-species-rare} for a rare observation. Table \ref{tab-stat-distance-three-species-common-100} shows the common observation situation for the larger copy number case of $Z(0)=(100,100,100)$.

One-stage targeting algorithm estimates the conditional distribution more accurately than the naive method. For the two-stage targeting algorithm, there is an optimal choice of $t_0$ to split the forward evolution stage and the targeting stage. When the length of the second interval $[t_0, T]$ is large, the growing disparity of the Girsanov weights could be a problem. However, if $t_0$ is very close to the final time $T$, we may be faced with very disparate Poisson weights $W$.
With the optimal choice of $t_0$, the two-stage targeting appears to perform better than the single stage targeting. 
On the other hand, within the two-stage algorithm, it appears that the using a common intensity $\overline{\lambda}$ and individual propensities tend to give similar performance.

\begin{table}[htbp]
    \centering
    \begin{adjustbox}{width=\textwidth,center=\textwidth}
    \begin{tabular}{|c| c | c| c | c | c | c | c| }
        \hline
        Method  &TVE & ESF $(W)$  & ESF $(\tilde{L})$ & ESF & Runtime\\
          &($95\%$ interval) &  & & &(seconds)\\
        \hline
        Naive  & $0.2146, [0.2117, 0.2174]$ & N/A & N/A & $0.16$ & $0.331$\\
        \hline
        One stage $\lambda_1$, $\Delta t = 0.1$  & $0.1695, [0.1669, 0.1722]$ & $0.83$ & $0.20$ & $0.21$ & $0.773$\\
        One stage $\lambda_2$, $\Delta t = 0.1$ & $0.1699, [0.1673, 0.1726] $ &  $0.84$ & $0.20$ & $0.21$ & $0.770$\\
        \hline
        Two stage, $\overline{\lambda}$, $t_0=0.5$ & $0.1375, [0.1346, 0.1404] $  & $0.80$ & $0.36$ & $0.35$ & $0.55$\\
        Two stage, $\overline{\lambda}$, $t_0=0.8$ & $0.1133, [0.1117, 0.1150]$ & $0.74$ & $0.66$ & $0.54$ & $0.57$\\
        Two stage, $\overline{\lambda}$, $t_0=0.9$ & $0.1113, [0.1098, 0.1128]$ & $0.74$ & $0.81$ & $0.63$ & $0.50$\\
        Two stage, $\overline{\lambda}$, $t_0=0.99$ & $0.1852, [0.1826, 0.1878]$ & $0.39$ & $0.94$ & $0.39$ & $0.49$\\
        Two stage, $\overline{\lambda}$, $t_0=0.999$ & $0.2555, [0.2519, 0.2590]$ & $0.29$ & $0.96$ & $0.29$ & $0.54$\\
        \hline
        
        Two stage, $\lambda^{(i)}$, $t_0=0.5$& $0.1616, [0.1576, 0.1655]$ & $0.79$ & $0.31$ & $0.28$ & $0.54$\\  
        Two stage, $\lambda^{(i)}$, $t_0=0.8$ & $0.1045, [0.1031, 0.1059]$ & $0.86$ & $0.69$ & $0.72$ & $0.52$\\         
        Two stage, $\lambda^{(i)}$, $t_0=0.9$ & $0.1025, [0.1010, 0.1039]$ & $0.82$  & $0.85$ & $0.78$ & $0.50$\\                
        Two stage, $\lambda^{(i)}$, $t_0=0.99$ & $0.1826, [0.1802, 0.1851]$ &  $0.40$ & $0.98$ & $0.40$ & $0.50$\\
        Two stage, $\lambda^{(i)}$, $t_0=0.999$ & $0.2559, [0.2525, 0.2594]$ & $0.29$ & $0.99$ & $0.29$ & $0.55$\\
        \hline
    \end{tabular}
    \end{adjustbox}
    \caption{Example 3. A common observation. We took  $Z(0)=(20,20,20)$, $c=(0.5, 1, 0.1, 1)$, $t = T = 1$,  $y_T = 24$. We used $\overline{\lambda}$ to denote common intensity and $\lambda^{(i)}$ for individual intensities in the two stage methods. We used $N_r = 1,000$ trials with filter sample size $N_s = 1,000$.}
    \label{tab-stat-distance-three-species-common}
\end{table}

\begin{table}[htbp]
    \centering
    \begin{adjustbox}{width=\textwidth,center=\textwidth}
    \begin{tabular}{|c| c | c| c | c | c | c | c| }
        \hline
        Method  &TVE & ESF $(W)$  & ESF $(\tilde{L})$ & ESF & Runtime\\
          &($95\%$ interval) &  & & &(seconds)\\
        \hline
        Naive  & $ 0.4888, [0.4741, 0.5034]
$ & N/A & N/A & $0.07$ & $0.64$\\
        \hline
        One stage $\lambda_1$, $\Delta t = 0.1$  & $0.2995, [0.2900, 0.3089]
$ & $0.80$ & $0.16$ & $0.16$ & $0.84$\\
        One stage $\lambda_2$, $\Delta t = 0.1$ & $0.2955, [0.2862, 0.3048]
$ &  $0.84$ & $0.16$ & $0.17$ & $0.81$\\
        \hline
        Two stage, $\overline{\lambda}$, $t_0=0.5$ & $0.2458, [0.2377, 0.2540]
$  & $0.74$ & $0.28$ & $0.23$ & $0.73$\\
        Two stage, $\overline{\lambda}$, $t_0=0.8$ & $0.1750, [0.1694, 0.1805]
$ & $0.72$ & $0.65$ & $0.50$ & $0.80$\\
        Two stage, $\overline{\lambda}$, $t_0=0.9$ & $0.1767, [0.1712, 0.1823]$ & $0.64$ & $0.81$ & $0.56$ & $0.79$\\
        Two stage, $\overline{\lambda}$, $t_0=0.99$ & $0.2361, [0.2287, 0.2435]$ & $0.41$ & $0.95$ & $0.41$ & $0.78$\\
        Two stage, $\overline{\lambda}$, $t_0=0.999$ & $0.4680, [0.4559, 0.4800]$ & $0.17$ & $0.94$ & $0.17$ & $0.83$\\
        \hline
        
        Two stage, $\lambda^{(i)}$, $t_0=0.5$& $0.2549, [0.2427, 0.2671]
$ & $0.77$ & $0.28$ & $0.24$ & $0.74$\\  
        Two stage, $\lambda^{(i)}$, $t_0=0.8$ & $ 0.1561, [0.1516, 0.1606]$ & $0.85$ & $0.74$ & $0.72$ & $0.78$\\         
        Two stage, $\lambda^{(i)}$, $t_0=0.9$ & $0.1579, [0.1535, 0.1623]$ & $0.76$  & $0.90$ & $0.72$ & $0.77$\\                
        Two stage, $\lambda^{(i)}$, $t_0=0.99$ & $0.2363, [0.2296, 0.2430]$ &  $0.42$ & $0.98$ & $0.42$ & $0.79$\\
        Two stage, $\lambda^{(i)}$, $t_0=0.999$ & $0.4689, [0.4548, 0.4831]
$ & $0.17$ & $0.99$ & $0.17$ & $0.81$\\
        \hline
    \end{tabular}
    \end{adjustbox}
    \caption{Example 3. A common observation. We took $Z(0)=(100, 100, 100)$, $c=(0.5, 1, 0.02, 1)$, $t = T = 1$,  $y_T = 120$. We used $\overline{\lambda}$ to denote common intensity and $\lambda^{(i)}$ for individual intensities in the two stage methods. We used a filter sample size of $N_s = 1,000$ and ran $N_r=100$ independent trials.}
    \label{tab-stat-distance-three-species-common-100}
\end{table}

\begin{table}[htbp]
    \centering
    \begin{adjustbox}{width=\textwidth,center=\textwidth}
    \begin{tabular}{|c| c | c| c | c | c | c | c| }
        \hline
        Method  &TVE & ESF $(W)$  & ESF $(\tilde{L})$ & ESF & Runtime\\
          &($95\%$ interval) & & &  &(seconds)\\
        \hline
        Naive  & $0.4347, [0.4286, 0.4408] $ & N/A & N/A & $0.04$ & $0.326$\\
        \hline
        One stage $\lambda_1$, $\Delta t = 0.1$& $0.2184, [0.2152, 0.2217]$ &  $0.70$ & $0.12$ & $0.14$ & $0.791$\\
        One stage $\lambda_2$, $\Delta t = 0.1$ & $0.2103, [0.2071, 0.2135]$ & $0.84$ &
        $0.14$ & $0.15$ & $0.794$\\
        \hline
        Two stage, $\overline{\lambda}$, $t_0=0.5$ & $ 0.1677, [0.1653, 0.1701]$ & 
        $0.64$ &  $0.24$ & $0.25$ & $0.59$\\
        Two stage, $\overline{\lambda}$, $t_0=0.8$  & $0.1627, [0.1601, 0.1654] $ & 
        $0.50$ & $0.46$ & $0.31$ & $0.50$\\
        Two stage, $\overline{\lambda}$, $t_0=0.9$ & $ 0.1885, [0.1860, 0.1909]$ & $0.34$ & $0.66$ & $0.25$ & $0.56$\\
        Two stage, $\overline{\lambda}$, $t_0=0.99$ & $0.3963, [0.3911, 0.4015]$ & $0.08$ & $0.92$ & $0.07$ & $0.53$\\
        Two stage, $\overline{\lambda}$,$t_0=0.999$ & $0.5545, [0.5472, 0.5617]$ & $0.05$ & $0.96$ & $0.05$ & $0.54$\\
        \hline
        Two stage, $\lambda^{(i)}$, $t_0=0.5$ & $0.2115, [0.2070, 0.2160]$ & $0.64$  & $0.19$ & $0.19$ & $0.58$\\         
        Two stage, $\lambda^{(i)}$, $t_0=0.8$  & $0.1641, [0.1619, 0.1663]$ & $0.52$ & $0.49$ & $0.32$ & $0.51$\\        
        Two stage, $\lambda^{(i)}$, $t_0=0.9$ & $0.1890, [0.1866, 0.1914]$ & $0.33$ & $0.68$ & $0.26$ & $0.50$\\        
        Two stage, $\lambda^{(i)}$,$t_0=0.99$ & $0.3940, [0.3887, 0.3992]$ & $0.07$  & $0.89$ & $0.07$ & $0.53$\\
        Two stage, $\lambda^{(i)}$, $t_0=0.999$ & $0.5662, [0.5585, 0.5739]$ & $0.05$ & $0.92$ & $0.05$ & $0.56$\\        
       \hline
    \end{tabular}
    \end{adjustbox}
    \caption{Example 3. A rare observation. We took  $Z(0)=(20,20,20)$, $c=(0.5, 1, 0.1, 1)$, $t = T=1$, $y_T = 20$. We used $\overline{\lambda}$ to denote common intensity and $\lambda^{(i)}$ for individual intensities in the two stage methods. We used $N_r = 1,000$ trials, with filter sample size $N_s = 1,000$.}
    \label{tab-stat-distance-three-species-rare}
\end{table}

\subsection{Comparison with using the conditional propensity}
Here, we compare the targeting algorithm with the idea of using the conditional propensity function (when available) 
to generate $\tilde{Z}$ as discussed in section \ref{sec-cond-prop} (also see \cite{golightly2019efficient}). 
We use Examples 1 and 2 since in those two examples, the conditional propensities can be computed as described in Section \ref{sec-examples}. However, as mentioned there, only in Example 1 it is possible to obtain samples from the non-exponential wait times per the conditional propensity. In Example 2, sampling from the non-exponential wait time is not practical. An alternative to sampling from the
non-exponential wait time is to approximate the conditional propensity by holding its value constant at the most recent jump time, 
a strategy used in \cite{golightly2019efficient} which makes the wait times exponential.


In Example 1, as Table \ref{tab-eg1} shows, the exact time-dependent conditional propensity method (CP exact) has the best performance, as expected. However, our targeting method significantly outperformed the approximate conditional propensity method (CP approx) described above. In Example 2, as mentioned earlier, only the ``CP approx'' method is practical. For the small total copy number case, as shown in Table \ref{tab-eg2-small}, ``CP-approx'' method slightly outperforms our targeting methods in the rare observation case while slightly underperforms our targeting methods in the common observation case. However, the computational time is much longer than the targeting methods. Moreover, in Example 2, for the modestly large copy number case the ``CP-approx'' method is computationally prohibitive and hence is not shown in Table \ref{tab-eg2-large-sys}.

\subsection{Resampling}\label{sec-resample}
Since the disparity of the weights will usually grow fast as time progresses, we could use the resampling technique to alleviate the issue \cite{bain2008fundamentals, rathinam2021state}. Example 2 with a larger final time $T=10$ is used to illustrate this. 
We resampled after intervals of length $\Delta s$ and for convenience we took $\Delta s = \Delta t = 0.25$, that is equal to the 
piecewise constant intervals of the intensity. See Table \ref{tab-eg2-resample}, where it is clear (from the expected TVE) that  resampling improves the performance. However, the resampling procedure itself introduces extra randomness which could hurt the accuracy on the other hand. The frequency of resampling is an interesting question to explore as in particle filtering method.



\begin{table}[htbp]
    \centering
    \begin{adjustbox}{width=1.2\textwidth,center=\textwidth}
    \begin{tabular}{|c| c | c| c | c | c|  c |c | c |}
        \hline
        Method & Observation &TVE  & ESF ($W)$ & ESF ($\tilde{L}$) & ESF &Runtime\\
         &  &($95\%$ interval) &   &  & & (seconds)\\
        \hline
        Naive &  $y_T=80$  & $ 0.4178, [0.4058, 0.4299] $ & N/A  &  N/A & $0.059$ & $11.1$\\
        Targeting $(\overline{\lambda}_1)$, without resampling & (common observation)  & $1.1456, [1.1126, 1.1786]$ & $0.861$ &  $0.005$ & $0.005$ & $35.9$\\
        Targeting, resample $(\overline{\lambda}_1)$,  $\Delta s = 0.25$ &   & $0.1574, [0.1533, 0.1615]$ &  & & $0.948$ & $39.2$\\
        \hline
        
        \hline
        Naive & $y_T = 98$   &  $1.7246, [1.7000, 1.7492]$ & N/A & N/A & $0.002$ & $17.9$\\
        Targeting $(\overline{\lambda}_1)$, without resampling & (rare observation) & $1.3004, [1.2694, 1.3315] $ & $0.804$ & $0.004$ & $0.003$ & $27.5$\\
        Targeting, resample $(\overline{\lambda}_1)$, $\Delta s = 0.25$ & & $0.1830, [0.1779, 0.1881]$ & & & $0.779$ & $27.9$\\
        \hline
    \end{tabular}
    \end{adjustbox}
    \caption{Example 2. $S_1 \longleftrightarrow S_2$. We took $z_0 = (100, 100)$, $c=(1, 1.5)$, $T=10$, $t=9$, $\Delta t = 0.25$, while taking a common observation $y_T = 80$ as well as a rare observation $y_T =  98$. We used $N_r = 100$ trials, while each trial used filter sample size $N_s = 2,000$.  We note that in our experiment for the rare observation case, the naive method has  $98\%$ of the trials successfully producing a meaningful estimate without encountering a zero denominator issue.}
    \label{tab-eg2-resample}
\end{table}

\subsection{Increasing filter sample size}

Finally, we investigated the convergence 
of the errors (as measured by the expected total variation error) as the filter sample size $N_s$ increases. We focused on Example 3 and considered the naive method as well as four different variants of the targeting methods that were considered in Table \ref{tab-stat-distance-three-species-common}. We varied the filter sample size $N_s$ in the range of $1000 - 64000$. The plot of mean TVE (along with confidence intervals) in log scale against $1/\sqrt{N_s}$ is shown in Figure \ref{fig-convergence}. A reference line of slope $1$ is shown. All the methods seem to converge as $1/\sqrt{N_s}$. 

A few comments are in order. 
As explained in Section \ref{sec-ESS}, by the law of large numbers  
the error in computing the conditional probability $\pi(t,z)$ of a fixed state $z$ is expected to converge to zero under modest integrability assumptions. 
However, the rate of convergence $1/\sqrt{N_s}$ mentioned in Section \ref{sec-ESS} is based on weak convergence. On the other hand, numerical estimate of this rate corresponds to the convergence in the mean of the quantity $\sqrt{N_s} \, |\hat{\pi}_{N_s}(t,z)-\pi(t,z)|$ (thus an $\mathcal{L}^1$ convergence rate of $1/\sqrt{N_s}$). 
We do not have a convergence proof of this nature at this point in time. Several convergence results for particle filters exist in the literature, see \cite{crisan2002survey} for instance. These results apply in situations where the state space is the continuum $\real^d$ and the distributions have densities. We believe it is possible to adapt these techniques to prove convergence for our situation where the state space is the integer lattice. This is a subject for future work.

\begin{figure}
    \centering
    \includegraphics{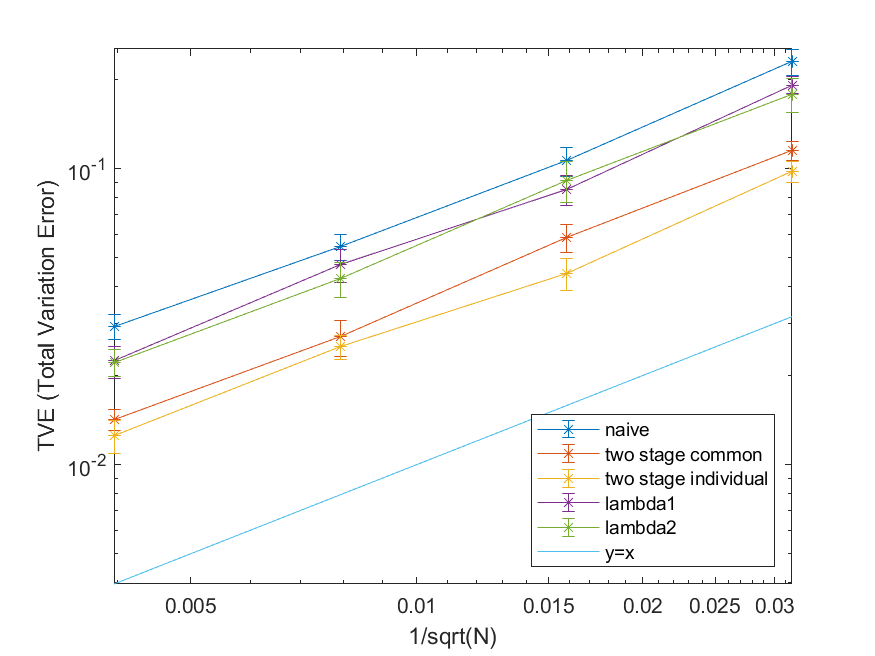}
    \caption{Convergence of various simulation methods as sample size increase. The Total Variation Error (TVE) against $\frac{1}{N_s}$ in log scale. We used Example 3 a common observation to illustrate. We took $Z(0)=(20, 20, 20)$, $c=(0.5, 1, 0.1, 1)$, $t = T = 1$, $\Delta t = 0.1$,  $y_T = 24$. We used $N_r = 100$ trials, filter sample sizes were $N_s = 1000, 4000, 16000, 64000$.}
    \label{fig-convergence}
\end{figure}

\section{Conclusion}\label{sec-conclusion}
In this paper we provided a new particle filtering method,
{\em the targeting algorithm}, for statistical inference in stochastic reaction networks where exact partial state observations are available in discrete time snapshots. Our targeting algorithm is designed so that the simulated process always satisfies the observations. This is accomplished by splitting the reaction channels into ``free'' and ``slaved'' ones and making use of the fact that the partial state observations impose linear constraints among the reaction counts. Moreover, under the simulation probability measure, in between observations, the reaction count processes are taken to be inhomogeneous Poissons to facilitate easy interpolation in between observations. 

We provided theoretical justification as well as numerical examples where the latter show superior performance compared to a prediction/correction approach.  
We also discussed and illustrated different choices within our targeting algorithm which include the selection of the inhomogeneous intensity for the proposals, the choice of free and slaved reactions and the two-stage targeting approach.  The optimal choice among these options is the subject of future research. While our work focused on the case of exact observations (which also covers Poisson type noise as explained in Section \ref{sec-intro}), our method could be adapted to the case of additive Gaussian noise considered by other researchers \cite{golightly2019efficient, fang2022stochastic}. In fact, the analysis provided in Appendix A suggests that, at least in the case of small additive noise, our targeting approach should perform much better than a prediction/correction approach. This is the subject of future investigations. 

\section*{Appendix A: Advantage of targeting over prediction/correction}
Here we consider a probability space $(\Omega,\sF,P_0)$ and an $m$-variate Poisson random variable $R(T)=(R'(T),R''(T))$ where $R'(T)$ is $\posint^{m_1}$ valued and $R''(T)$ is $\posint^{m_2}$ valued, with mean $M=(M',M'')$.
Let $C$ be a matrix and $d$ be a vector of appropriate dimensions such that the event $\{R''(T)=C R'(T)+d\}$ makes sense. We wish to generate an identically distributed weighted sample $(K^{(i)},W^{(i)})$ distributed like $(K,W)$ so that the conditional distribution
is given by
\[
P_0\{R(T)=k \, | \, R''(T)=C R'(T) +d\} = \frac{\E_0[1_{\{k\}}(K) W]}{\E_0[W]}.
\]
The targeting approach generates $K' \sim R'(T)$ and sets 
$K''= C K'+d$ and sets 
\[
W = W_p= \sum_{k'} P_0\{K''=C k' + d\} 1_{\{k'\}}(K').
\]
A more direct approach based on prediction/correction would be to generate $K \sim R(T)$ and set $W=W_i=1_{\{K''=C K'+d\}}$. We have used subscripts $p$ and $i$ to denote the Poisson weight and the indicator weight respectively. The effective sample fraction of a weight $W$ 
(w.r.t.\ probability measure $P_0$) could be taken to be $\text{ESF}=\E_0^2[W]/E_0[W^2]$. We obtain the following formulae for $\text{ESF}_p$ and $\text{ESF}_i$, the effective sample fractions of the Poisson weight and the indicator weight. First we note that 
\begin{equation}
\E_0[W_p]=\E_0[W_i]=P_0\{R''(T)=C R'(T)+d\},    
\end{equation}
so that both weights have the same mean which equals the probability of the event $\{R''(T)=C R'(T) + d\}$. 
On the other hand, we obtain 
\[
\E_0[W^2_i]=\E_0[W_i] = P_0\{R''(T)=C R'(T)+d\},
\]
where as
\[
\E_0[W_p^2] = \sum_{k'} P^2_0\{R''(T)=C k' + d\} P_0\{R'(T)=k'\}. 
\]
Let us denote by 
\[
\bar{\rho} = \max_{k'} P_0\{R''(T)=C k' +d\}.
\]
Then, we see that $\E_0[W^2_p] \leq \bar{\rho} P_0\{R''(T)=C R'(T) +d\}$. Consequently, $\text{ESF}_p \geq \text{ESF}_i/\bar{\rho}$. As an upper bound of $\bar{\rho}$ we may use the maximum of the p.m.f.\ of $R''(T)$ which is an $m_2$-variate Poisson. The maximum p.m.f.\ of a Poisson random variable with mean $M$ may be approximated by using the Stirling's approximation (which holds here for large $M$) to obtain 
$(2 \pi M)^{-1/2}$. Hence, we obtain the bound
\[
\text{ESF}_p \geq \text{ESF}_i (2 \pi)^{m_2/2} (M_{m_1+1} \cdots M_{m})^{1/2},
\]
which is valid when $M_i$ are sufficiently large. It is clear that the larger the product $M_{m_1+1} \cdots M_{m}$, the more efficient the targeting method would be over the prediction/correction approach. 
As an example, if only a single species is observed, then there is only one constraint among the reaction counts and hence $m_2=1$. If the mean reaction count is modestly large, say $10$, we expect an improvement in the effective sample fraction of $\sqrt{20 \pi} \approx 7.9$. Clearly, the improvement gets better if the number of observed species increases. 
\section*{Appendix B: Rigorous Treatment of the Targeting Algorithm}\label{sec-rigorous}

The targeting problem for a reaction network is stated as follows. We are given the following information about the lab process $Z$.
\begin{enumerate}
    \item An observed set of events $O_p$ from the past (before time $t=0$ which is referred to as ``present'') with $P(O_p)>0$. 
    \item The conditional probability $\mu$ of the lab process $Z$ at time $t=0$. Thus, for $z_0 \in \posint^{n}$
    \[
    \mu(z_0) = P\{ Z_0 = z_0\, | \, O_p\}.
    \]
    \item Future time $T>0$ and future observation $Y(T)=y$ 
    where $y \in \posint^{n_2}$.
\end{enumerate}
The goal is to compute the conditional distribution $\pi(t,z)$
 where for $t \in [0,T]$ and $z \in \posint^n$
 \[
 \pi(t,z) = P\{Z(T)=z \, | \, O_p, Y(T)=y\}.
 \]

The targeting algorithm consists of using an alternate probability measure $P_0$ such that under $P_0$ and under conditioning on $O_p$ and $\{Z_0=z_0\}$, the lab process $R$ for reaction counts is an inhomogeneous Poisson with a nonvanishing intensity $\lambda(t,z_0,y)>0$ on $[0,T]$. Then the filter processes $\tilde{Z}$ and $I$ (proxies for $Z$ and $R$) are created such that, under $P_0$ and when conditioned on $O_p$ and $\{\tilde{Z}_0=z_0\}$, $I$ is an inhomgeneous Poisson with a nonvanishing intensity $\lambda(t,z_0,y)>0$ on $[0,T]$.

We remark that it is conceptually easier to think of the role 
of the future observation $y$ as fixed. Thus, for a given targeting problem, $y$ is fixed and the probability measure $P_0$ depends on $y$, thus $P_0=P_0^y$. We shall suppress the superscript $y$. The desired probability measure $P=P^y$ under which $R$ has the correct intensity is accomplished by a 
Girsanov change of measure: $dP = L(T) dP_0$. Since the filter process $\tilde{Z}$ will be used for the estimation, the probability measure to use shall $d \tilde{P}=\tilde{L}(T) dP_0$.  

\subsection*{B.1: The Lab Process}
We start with $(\Omega,\sF,P_0)$, a probability space. 
Let $T>0$, $X_0:\Omega \to \posint^{n_1}$ and $Y_0:\Omega \to \posint^{n_2}$ be random variables. Define $Z_0=(X_0,Y_0)$. Let $\nu \in \integ^{n \times m}$.

Let $R:[0,T] \times \Omega \to \posint^m$ be an $m$-variate counting process 
which is broken into $R=(R',R'')$ where $R'$ and $R''$ are $m_1$-variate and $m_2$-variate respectively. 
This means for $i=1,\dots,m$, $R_i$ is a {\em cadlag} process which is non-decreasing, takes values in $\posint$ and $R_i(0)=0$. 

Define processes $X,Y$ by 
\[
\begin{aligned}
X(t) &= X_0 + \nu' R(t),\\
Y(t) &= Y_0 + \nu'' R(t).\\
\end{aligned}
\]
The $\integ^n$ valued process $Z$ is defined by $Z(t)=(X(t),Y(t))$. Thus we may write
\[
Z(t) = Z_0 + \nu R(t).
\]
We let $\sY_p \subset \sF$ denote a sub $\sigma$-algebra of observation events that ``happened prior'' to time $t=0$ and we assume $O_p \in \sY_p$ is a set of ``past observations'' with $P_0(O_p)>0$. 
We define 
\[
\sF_t = \sY_p \bigvee \sigma(Z(s), R(s) \, | \, 0 \leq s \leq t). 
\]
Recall that we may assume that with suitable  splitting of the reactions $R(t)=(R'(t),R''(t))$, 
we may write $\nu''$ as 
\begin{equation}
    \nu'' = [A \; B]
\end{equation}
where $B$ is invertible. This yields
\[
Y(T)=Y_0 + A R'(T) + B R''(T)
\]
and hence 
\[
R''(T) = B^{-1} (Y(T)-Y_0-A R'(T)).
\]
For notational convenience we define $G:\integ^{n_2} \times\posint^{m_1} \to \posint^{m_2}$ by 
\[
G(\Delta y, k') = B^{-1} (\Delta y - A k')
\]
for $\Delta y \in \integ^{n_2}$ and $k' \in \posint^{m_1}$. For $\Delta y \in \integ^{n_2}$ 
we also define the subset $S_{\Delta y} \subset \posint^{m_1}$ by 
\[
S_{\Delta y} = \{ k' \in \posint^{m_1} \, | \, B^{-1}(\Delta y - A k') \in \posint^{m_2} \}.
\]
For $\Delta y \in \integ^{n_2}$ and $k=(k',k'') \in \posint^{m_1} \times \posint^{m_2}$ the following equivalence holds: 
\[
\Delta y = \nu'' k  \Longleftrightarrow  k' \in S_{\Delta y} \text{ and } k'' = G(\Delta y,k').
\]
Since $Y(T)=Y_0 + \nu'' R(T)$, we note that for $y \in \integ^{n_2}$
\[
\{Y(T)-Y_0=y\} \subset \{R'(T) \in S_y\}.
\]

We are given $\lambda:[0,T] \times \posint^{n} \times \posint^{n_2} \to (0,\infty)^m$, a positive function. We are also given $\mu$, be a probability measure on $\posint^{n}$.

We suppose that the following hold under $P_0$: 
\begin{enumerate}
\item Conditioned on $O_p$, $Z_0=(X_0, Y_0)$ has the distribution $\mu$. 
\item Conditioned on $O_p$ and on $Z_0=z_0=(x_0,y_0)$, $R$ is an $m$-variate inhomogeneous Poisson process with intensity $\lambda(t,z_0,y)$. 
\end{enumerate}

We also split $\lambda$ into $\lambda=(\lambda',\lambda'')$ where $\lambda'$ is $\real^{m_1}$ valued and 
$\lambda''$ is $\real^{m_2}$ valued. 

Later on we shall consider another probability measure $P$ on $(\Omega, \sF)$ under which $R$ is an $m$-variate counting process with $\sF_t$-intensity $a(Z(t-))$. Thus, under $P$, the processes $R$ and $Z$ represent the reaction network. 
Since $\lambda$ is strictly positive, $P$ will be absolutely continuous with respect to $P_0$. 

\begin{remark}
When we refer to an $m$-variate Poisson process or an $m$-variate Poisson random variable with an $m$-vector valued intensity or mean, we assume that the components are independent and have intensity or mean equal to the corresponding component.
\end{remark}



\subsection*{B.2: Construction of filter random variables $\tilde{Z}_0=(U_0,V_0)$, $K$ and $W_p$}


On $(\Omega,\sF,P_0)$ we create a sequence $(U_{0,i}, V_{0,i},K'_i)$ which is $\posint^{n_1} \times \posint^{n_2} \times \posint^{m_1}$ with the following properties (under $P_0$). 

\begin{enumerate}
\item Conditioned on $O_p$, $(X_0, Y_0)$ and $(U_{0,i}, V_{0,i})$ (for $i \in \nat$) are all independent, and have the same distribution $\mu$. 
\item Conditioned on $O_p$, $K_i'$ for $i \in \nat$ 
is an iid sequence that is independent of $X_0, Y_0$ and $R$. 
\item For each $i$, when conditioned on $O_p$ 
and on $(U_{0,i}, V_{0,i})=(x_0, y_0)$, 
$K'_i$ is an $m_1$-variate Poisson random variable with mean $\int_0^T \lambda'(t, x_0, y_0,y) dt$. Thus for $k' \in \posint^{m_1}$ and $(x_0,y_0) \in \posint^n$
\[
P_0\{K'_i=k' \, | \, U_{i,0}=x_0, V_{i,0}=y_0, O_p\} = P_0\{R'(T)=k' \, | \, X_0=x_0, Y_0=y_0, O_p\}.
\]
\end{enumerate}

Define $\rho:\posint^{m_2} \times \posint^{n_1} \times \posint^{n_2} \times \posint^{n_2} \to \real$ as follows. For each $x_0, y_0, y$, $\rho( \cdot, x_0, y_0, y)$ is the p.m.f.\ of an $m_2$-variate Poisson random variable with mean $\int_0^T \lambda''(t,x_0, y_0,y) dt \in \real^{m_2}$. 

Define the random variables $N, K', K'', U_0, V_0$ and $W_p$ as follows. 
Define $N$ by
\[
N = \min \{ i \in \nat \, | \, K'_i \in S_{y-V_{0,i}} \}.
\]
Throughout this section we make the following (sensible) assumption.
\begin{assumption}\label{ass-feasible} 
$P\{Y(T)=y\}>0$. 
\end{assumption}
Since $P$ is absolutely continuous with respect to $P_0$, it follows that $P_0\{Y(T)=y\}>0$ and hence there exists a $y_0$ with $P_0\{Y_0=y_0\}>0$ such that $S_{y-y_0}$ is nonempty.
Therefore
\[
P_0\{R'(T) \in S_{y-y_0}\} >0.
\]
Hence, it follows that
\begin{equation*}
\begin{split}
&P_0\{R_i' \in S_{y-V_{0,i}}\} = \sum_{\tilde{y_0}} P_0\{R_{0,i}'(T) \in S_{y-\tilde{y_0}}\,|\, V_0=\tilde{y_0}\} P_0\{V_0=\tilde{y_0}\} \\
& > P_0\{R'(T) \in S_{y-{y_0}}\,|\, Y_0={y_0}\} P_0\{Y_0={y_0}\}>0
\end{split}
\end{equation*}
and thus
\[
P_0\{N < \infty\} = 1.
\]
Thus under Assumption \ref{ass-feasible}, $N$ is finite and $K'$, $K''$ and $W_p$ below are well defined:  
$K' = R'_N$, $K''=G(y-V_{0,N},K')$, $U_0 = U_{0,N}, V_0 = V_{0,N}$, $W_p=\rho(K'',U_0, V_0,y)$.  
Let $K=(K',K'')$.  We let $\tilde{Z}_0=(U_0,V_0)$. 


We shall denote the observed event 
\[
O_{y} = \{Y(T)=y\},
\]
and define $O=O_y \cap O_p$ (past and future observations).
We note that if there is an observation $Y_0=y_0$ at $t=0$ 
then it is automatically captured in this analysis by taking 
the $\sum_{x_0} \mu(x_0,y_0)=1$ and $\mu(x_0, \tilde{y}_0)=0$ 
for $\tilde{y}_0 \neq y_0$. 

\begin{lemma}\label{lem-Ri'-N-YT}
For $i \in \nat$, $z_0=(x_0,y_0) \in \posint^{n}$ and $k' \in \posint^{m_1}$ 
\[
P_0\{U_{0,i} =x_0, V_{0,i}=y_0, K_i' = k' \, | \, N=i, O\} = 1_{S_{y-y_0}}(k') \frac{ P_0\{X_0=x_0, Y_0=y_0, R'(T) = k' \, | \, O_p\}}{P_0\{R'(T) \in S_{y-Y_0} \, | \, O_p\}}.
\]
\end{lemma}
\begin{proof}
If $k' \notin S_{y-y_0}$ then both sides of the equation in the lemma are zero. Suppose $k' \in S_{y-y_0}$. 
Then we have that
\[
\begin{aligned}
&P_0\{U_{0,i}=x_0, V_{0,i} = y_0, K_i'=k' \, | \, N=i, O\}\\
&=P_0\{U_{0,i}=x_0, V_{0,i}=y_0, K_i'=k' \, | \, K'_i \in S_{y-V_{0,i}}, K'_j \notin S_{y-V_{0,j}} \text{ for } j<i, O_y, O_p\}\\
&=\frac{P_0\{U_{0,i}=x_0, V_{0,i}=y_0,  K_i'=k', K'_i \in S_{y-y_0}, K'_j \notin S_{y-V_{0,j}} \text{ for } j<i, O_y \, | \, O_p\}}{P_0\{K'_i \in S_{y-V_{0,i}}, K'_j \notin S_{y-V_{0,j}} \text{ for } j<i, O_y \, | \, O_p \}}\\
&=\frac{P_0\{U_{0,i}=x_0, V_{0,i}=y_0,  K_i'=k' \, | \, O_p\}}{P_0\{K'_i \in S_{y-V_{0,i}} \, | \, O_p\}}.\\
\end{aligned}
\]
We note that the third equality follows since the other events in the numerator are independent of $Y(T)$ (and thus of $O_y$)  when conditioned on $O_p$.  
Moreover, 
when conditioned on $O_p$, $(U_{0,i}, V_{0,i})$ is distributed like $(X_0, Y_0)$, and conditioned on $(U_{0,i}, V_{0,i})=(x_0,y_0)$ and $O_p$, $K'_i$ is distributed like $R'(T)$ when conditioned on $(X_0,Y_0)=(x_0,y_0)$ and $O_p$. Thus, we have $$P_0\{U_{0,i}=x_0, V_{0,i}=y_0,  K_i'=k' \, | \, O_p\}= P_0\{X_0=x_0, Y_0=y_0, R'(T) = k' \, | \, O_p\}.$$ Also, $P_0\{K'_i \in S_{y-V_{0,i}} \, | \, O_p\} = P_0\{R'(T) \in S_{y-Y_0} \, | \, O_p\}$. These observations clinch the result.
\end{proof}

\begin{proposition}\label{prop1}
For $z_0=(x_0,y_0) \in \integ^{n}, k' \in \posint^{m_1}$ we have that
\[
P_0\{\tilde{Z}_0=z_0, K'=k' \, | \, O\}
= 1_{s_{y-y_0}}(k') \, \frac{P_0\{Z_0=z_0, R'(T) = k' \, | \, O_p\}}{P_0\{R'(T) \in S_{y-Y_0} \, | \, O_p\}}.
\]
\end{proposition}
\begin{proof}
We may write
\[
P_0\{U_0=x_0, V_0=y_0, K'=k' \, | \, O \}
= \sum_{i \in \nat} P_0\{U_{0,i}=x_0, V_{0,i}=y_0, K'=k' \, | \, N=i, O \} P_0\{N=i \, | \, O \}.
\]
Applying Lemma \ref{lem-Ri'-N-YT} yields the result.
\end{proof}

\begin{proposition}\label{prop2}
For $z_0=(x_0,y_0) \in \integ^{n}, k \in \posint^m$ we have that
\[
P_0\{Z_0=z_0, R(T)=k \, | \, O \} = \frac{\E_0[1_{\{(z_0, k)\}}(\tilde{Z}_0, K) \, W_p \, | \, O]}{\E_0[W_p \, | \, O]}.
\]
\end{proposition}
\begin{proof}
Let $k=(k',k'')$ where $k' \in \posint^{m_1}, k'' \in \posint^{m_2}$.
If $y \neq y_0 + \nu'' k$ then both sides of the equation 
in Proposition \ref{prop2} are zero. 

Suppose $y=y_0 + \nu'' k$. Then $k''=G(y-y_0,k')$, $k' \in S_{y-y_0}$ and 
the event equality $\{K=k\}=\{K'=k'\}$ holds. Hence
\[
\E_0[ 1_{\{(z_0, k)\}}(\tilde{Z}_0,K) \, W_p \, | \, O] = \frac{\E_0[ W_p \, ; \,\tilde{Z}_0=z_0, K'=k', O]}{P_0\{O\}}.
\]
We note that on $\{K'=k', \tilde{Z}_0=z_0\}$ 
\[
W_p = \rho(k'',x_0, y_0,y) = P_0\{R''(T)=k''\,|\, Z_0 = z_0\} = P_0\{R''(T)=k''\,|\, Z_0 = z_0, O_p\} .
\]
Hence 
\[
\begin{aligned}
\E_0[ &1_{\{(z_0,k)\}}(\tilde{Z}_0, K) \, W_p \, | \,O]\\ &= P_0\{R''(T)=k''\,|\, Z_0 = z_0\} \, P_0\{\tilde{Z}_0 = z_0, K'=k' \, | \, O\}\\
&= P_0\{R''(T)=k''\,|\, Z_0 = z_0\} \, \frac{P_0\{Z_0=z_0, R'(T)=k' \, | \, O_p\}}{P_0\{R'(T) \in S_{y-Y_0} \, | \, O_p\}}\\
&=  P_0\{R''(T)=k''\,|\, Z_0 = z_0\}  \, \frac{P_0\{ R'(T)=k' \,|\,Z_0=z_0 \} P_0\{Z_0=z_0 \, | O_p\}}{P_0\{R'(T) \in S_{y-Y_0} \, | O_p\}}\\
&= \frac{P_0\{R'(T)=k', R''(T)=k''\,|\, Z_0 = z_0\}P_0\{Z_0=z_0\, | \ O_p\}}{P_0\{R'(T) \in S_{y-Y_0} \, | \, O_p\}}\\
&= \frac{P_0\{R(T)=k,  Z_0 = z_0 \, | \, O_p\}}{P_0\{R'(T) \in S_{y-Y_0} \, | \, O_p\}}
\end{aligned}
\]
where there second equality follows from Proposition \ref{prop1} (noting that $k' \in S_{y-y_0}$) and the fourth equality follows from the conditional independence of $R'(T)$ and $R''(T)$ given $Z_0$.  We have also used the independence of $(R'(T),R''(T))$ from $O_p$ when conditioned on $Z_0=z_0$.  

Thus for all $k \in \posint^{m}$ we may write
\[
\E_0[ 1_{\{(z_0, k)\}}(\tilde{Z}_0, K) \, W_p \, | \,O] = 1_{\{y-y_0\}}(\nu'' k) \, \frac{P_0\{Z_0=z_0, R(T)=k \, | \, O_p\}}{P_0\{R'(T) \in S_{y-Y_0} \, | \, O_p\}}.
\]
Since 
\[
1_{\{y-y_0\}}(\nu'' k) 1_{\{z_0\}}(Z_0) 1_{\{k\}}(K) = 1_{\{y\}}(Y(T) 1_{\{z_0\}}(Z_0) 1_{\{k\}}(K),
\]
it follows that
\begin{equation}\label{eq1x0kw}
\E_0[ 1_{\{(z_0, k)\}}(\tilde{Z}_0, K) \, W_p \, | \,O] = \frac{P_0\{Z_0=z_0, R(T)=k , Y(T)=y\, | \, O_p\}}{P_0\{R'(T) \in S_{y-Y_0} \, | \, O_p\}}.
\end{equation}
Summing over $(z_0,k)$ we obtain 
\begin{equation}\label{eq-W-O}
\E_0[W_p \, | \, O] = \frac{P_0\{Y(T)=y \, | \, O_p\}}{P_0\{R'(T) \in S_{y-Y_0} \, | \, O_p\}}.
\end{equation}
Dividing \eqref{eq1x0kw} by \eqref{eq-W-O} we obtain
\[
\begin{aligned}
\frac{\E_0[1_{\{(z_0, k)\}}(\tilde{Z}_0, K) \, W_p \, | \, O]}{\E_0[W_p \, | \, O]} &= \frac{P_0\{Z_0=z_0, R(T)=k, O_y \, | \, O_p\}}{P_0\{O_y \, | \, O_p\}}\\
&= P_0\{ Z_0=z_0, R(T)=k \, | \, O\}.
\end{aligned}
\]

\end{proof}

\subsection*{B.3: Construction of the proxy processes $I,\tilde{Z}=(U,V)$}

Having constructed $\tilde{Z}_0=(U_0,V_0)$ and  $K$, we define the process $I:[0,T] \times \Omega \to \posint^m$ as follows. Conditioned on $\tilde{Z}_0=z_0$, we let $\eta_i(t) = \int_0^t \lambda(s,z_0,y) ds$ for $t \in [0,T]$. Note that $\eta_i$ are strictly increasing functions. 
Let $\xi^i_l$ for $i=1,2,\dots,m$ and $l=1,2,\dots$ be an i.i.d.\ collection 
of random variables uniformly distributed on $[0,\eta_i(T)]$ and independent of $O_p, Z_0, R, \tilde{Z}_0$ and $K$. 
Let $t^i_l = \eta_i^{-1}(\xi^i_l)$ 
Then we set
\[
I_i(t) = \sum_{l=1}^{K_i} 1_{[t^i_l,\infty)}(t).
\]
The process $I$ is a proxy for $R$ and by construction $I(T)=K$. We define $\tilde{Z}=(U,V)$ by
\[
\tilde{Z}(t) = \tilde{Z}_0 + \nu I(t).
\]

Let us denote the restriction of $R$ to $[0,T]$ also by $R$. We note that the conditional law of $R$ given $R(T)=k,Z_0=z_0,O$ is the same as the conditional law of $I$ given $I(T)=K=k,\tilde{Z}_0=z_0, O$, since both will be ``inhomogeneous Poisson bridge'' processes with intensity $\lambda(t,z_0,y)$  under the said conditioning. We summarize this as the following proposition. 
\begin{proposition}\label{prop3}
Let $\phi:D_{\real^m}[0,T] \to [0,\infty)$ be a nonnegative measurable map. Then for $k \in \posint^m, z_0 \in \integ^{n}$, we have that
\[
\begin{aligned}
\E_0[\phi(R) \, &| \, R(T)=k, Z_0=z_0, O] 
= \E_0[\phi(I) \, | \, I(T)=k, \tilde{Z}_0=z_0, O].
\end{aligned}
\]
\end{proposition}

\begin{proposition}\label{prop4}
Let $\phi:D_{\real^m}[0,T] \times \integ^{n} \to [0,\infty)$ be a nonnegative measurable map. Then 
\[
\E_0[\phi(R,Z_0) \, | \, O] = \frac{\E_0[\phi(I,\tilde{Z}_0) \, W_p \, | \, O]}{E_0[W_p \, | \, O]}.
\]
\end{proposition}
\begin{proof}
\[
\begin{aligned}
&\E_0[\phi(R,Z_0) \, | \, O]\\ 
&= \sum_{z_0 \in \integ^{n}}  \sum_{k \in \posint^m} \E_0[\phi(R,z_0) \, | \, R(T)=k, Z_0=z_0, O] \, P_0\{R(T)=k, Z_0=z_0 \, | \, O\}\\
&=\sum_{k, z_0} \E_0[\phi(I,z_0) \, | \, K=k, \tilde{Z}_0=z_0, O] \, \frac{\E_0[1_{\{(z_0, k)\}}(\tilde{Z}_0, K) \, W_p \, | \, O]}{\E_0[W_p \, | \, O]},
\end{aligned}
\]
where in the second equality we have used Propositions \ref{prop2} and \ref{prop3}. 
Now 
\[
\begin{aligned}
&\frac{\E_0[1_{\{(z_0,k)\}}(\tilde{Z}_0, K) \, W_p \, | \, O]}{\E_0[W_p \, | \, O]} = \frac{\E_0[W_p; K=k, \tilde{Z}_0=z_0, O]}{\E_0[W_p; O]}\\
&= \frac{\E_0[W_p \, | \, K=k, \tilde{Z}_0=z_0, O] \, P_0\{K=k, \tilde{Z}_0=z_0, \, | \, O\}}{\E_0[W_p \, | \, O]}.
\end{aligned}
\]
Since the process $I$ and $W_p$ are independent when conditioned on $I(T)=k, \tilde{Z}_0=z_0, O$, 
\[
\begin{aligned}
&\E_0[\phi(I, z_0) \, | \, K=k, \tilde{Z}_0=z_0, O]\, \E_0[W_p \, | \, K=k, \tilde{Z}_0=z_0, O]\\ &= \E_0[\phi(I, z_0) \, W_p\, | \, K=k, \tilde{Z}_0=z_0, O].
\end{aligned}
\]
Hence
\[
\begin{aligned}
&\E_0[\phi(R,Z_0) \, | \, O] \\ 
&=\frac{\sum_{k, z_0} \E_0[\phi(I,z_0) \, W_p \, | \, K=k, \tilde{Z}_0=z_0, O] \, P_0\{K=k, \tilde{Z}_0=z_0\, | \,O\}}{E_0[W_p \, | \, O]}\\
&= \frac{\E_0[\phi(I,\tilde{Z}_0) \, W_p \, | \, O]}{E_0[W_p \, | \, O]}.\\
\end{aligned}
\]
\end{proof}

\subsection*{B.4: Girsanov change of measure}

Following \cite{bremaud}, we define the processes $L_j,\tilde{L}_j:[0,T] \times \Omega \to \real$ for $j=1,\dots,m$
by
\begin{equation}
\begin{aligned}
L_j(t) &= \left(\prod_{i \geq 1} \frac{a_j(Z(T^j_i-))}{\lambda_j(T^j_i-,Z_0,y)} 1_{\{T^j_i \leq t\}}\right) \, \exp \left( \int_0^t (\lambda_j(s, Z_0, y) - a_j(Z(s)) \, ds \right),\\
\tilde{L}_j(t) &= \left(\prod_{i \geq 1} \frac{a_j(\tilde{Z}(\tilde{T}^j_i-)}{\lambda_j(\tilde{T}^j_i-,\tilde{Z}_0,y)} 1_{\{\tilde{T}^j_i \leq t\}}\right) \, \exp \left( \int_0^t (\lambda_j(s,\tilde{Z}_0,y) - a_j(\tilde{Z}(s)) \, ds \right),\\
\end{aligned}
\end{equation}
where $T^j_i$ and $\tilde{T}^j_i$ for $i=1,2,\dots$ 
are the $i$th jump time of $R_j$ and $I_j$ respectively. 
Define the processes $L,\tilde{L}$ by
\begin{equation}
L(t) = L_1(t) \dots L_m(t), \quad \tilde{L}(t) = \tilde{L}_1(t) \dots \tilde{L}_m(t).    
\end{equation}

We note that for each $t \in [0,T]$, we can write $L(t)$ and $\tilde{L}(t)$ as 
\[
L(t) = \psi_t(R, Z_0) \text{ and } \tilde{L}(t) = \psi_t(I,\tilde{Z}_0),
\]
where $\psi_t:D_{\real^m}[0,T] \times \integ^{n} \to \real$ is 
a measurable map. 

Under certain integrability conditions we may assume that $\E_0(L(T))=1$. 
Let $\sF_t=\sigma(Z_0, R(s);0\leq s \leq t)$.
It follows that $L$ 
is a $(\{\sF_t\}, P_0)$ martingale and $\E_0(L(t))=1$ for all $t \in [0,T]$. 
 We define the probability measure $P$ on $(\Omega,\sF)$ by 
$P=L(T)\,P_0$. It follows that 
the $P$ intensity of $R$ is 
$a(Z_0+\nu R(t-))=a(Z(t-))$,   
which is the desired intensity \cite{bremaud}.

\begin{theorem}
Let $\phi:D_{\real^m}[0,T] \times \posint^n \to \real$ be nonnegative measurable. Then
\[
\E[\phi(R, Z_0) \, | \, O] = \frac{\E_0[\phi(I, \tilde{Z}_0) \, \tilde{L}(T) \, W_p| \, O]}{\E_0[\tilde{L}(T) \, W_p| \, O]}
\]
\end{theorem}
\begin{proof}
\[
\begin{aligned}
&\E[\phi(R, Z_0) \, | \, O] =  \frac{\E[\phi(R, Z_0); \, O]}{P\{O\}}\\
&= \frac{\E_0[\phi(R, Z_0) \, L(T) ; \, O]}{\E_0[L(T) \,; \, O]} = \frac{\E_0[\phi(R, Z_0) \, L(T) \, | \, O]}{\E_0[L(T) \,| \, O]}.
\end{aligned}
\]
Since on $O$, $L(T)=\psi_T(R,Z_0)$ and $\tilde{L}(T)=\psi_T(I,Z_0)$, by Proposition \ref{prop4}, it follows that 
\[
\E_0[\phi(R, Z_0) \, L(T) \, | \, O] = \frac{\E_0[\phi(I,\tilde{Z}_0) \, \tilde{L}(T) \, W_p| \, O]}{\E_0[ W_p| \, O]},
\]
and 
\[
\E_0[L(T) \, | \, O] = \frac{\E_0[\tilde{L}(T) \, W_p| \, O]}{\E_0[W_p| \, O]}.
\]
From the last three equations the result follows. 
\end{proof}

\begin{corollary}
Let $z_0 \in \posint^{n}, k \in \posint^m$. Then for $t \in [0,T]$
\[
\E[1_{\{(z_0,k)\}}(Z_0,R(t)) \, | \, O] = \frac{\E_0[1_{\{(z_0,k)\}}(\tilde{Z}_0,I(t)) \, \tilde{L}(T) \, W_p| \, O]}{\E_0[\tilde{L}(T) \, W_p| \, O]}
\]
\end{corollary}

\begin{corollary}
For $z \in \posint^n$ and $t \in [0,T]$
\[
\E[1_{\{z\}}(Z(t)) \, | \, O] = \frac{\E_0[1_{\{z\}}(\tilde{Z}(t)) \, \tilde{L}(T) \, W_p| \, O]}{\E_0[\tilde{L}(T) \, W_p| \, O]}
\]
\end{corollary}


\bibliographystyle{elsarticle-num} 
\bibliography{cas-refs}





\end{document}